\documentclass[nointlimits]{amsart}

\usepackage[T1]{fontenc} 
\usepackage[utf8]{inputenc} 
\usepackage[USenglish]{babel} 

\usepackage{lmodern}
\usepackage{mathrsfs}
\usepackage{microtype}

\usepackage{url}\usepackage{natbib}

\usepackage{amsmath, amssymb, amsthm}

\usepackage{xcolor}

\usepackage{enumerate}

\usepackage{hyperref}
\newcommand{\TITLE}{Uniform decay of function norms}
\hypersetup{
	unicode=true,
	breaklinks=true,
	pdftitle={\TITLE},
	pdfauthor={Zdeněk Mihula, Maximilián Pándy}
	}
	
\makeatletter
\let\save@mathaccent\mathaccent
\newcommand*\if@single[3]{%
  \setbox0\hbox{${\mathaccent"0362{#1}}^H$}%
  \setbox2\hbox{${\mathaccent"0362{\kern0pt#1}}^H$}%
  \ifdim\ht0=\ht2 #3\else #2\fi
  }
\newcommand*\rel@kern[1]{\kern#1\dimexpr\macc@kerna}
\newcommand*\widebar[1]{\@ifnextchar^{{\wide@bar{#1}{0}}}{\wide@bar{#1}{1}}}
\newcommand*\wide@bar[2]{\if@single{#1}{\wide@bar@{#1}{#2}{1}}{\wide@bar@{#1}{#2}{2}}}
\newcommand*\wide@bar@[3]{%
  \begingroup
  \def\mathaccent##1##2{%
    \let\mathaccent\save@mathaccent
    \if#32 \let\macc@nucleus\first@char \fi
    \setbox\z@\hbox{$\macc@style{\macc@nucleus}_{}$}%
    \setbox\tw@\hbox{$\macc@style{\macc@nucleus}{}_{}$}%
    \dimen@\wd\tw@
    \advance\dimen@-\wd\z@
    \divide\dimen@ 3
    \@tempdima\wd\tw@
    \advance\@tempdima-\scriptspace
    \divide\@tempdima 10
    \advance\dimen@-\@tempdima
    \ifdim\dimen@>\z@ \dimen@0pt\fi
    \rel@kern{0.6}\kern-\dimen@
    \if#31
      \overline{\rel@kern{-0.6}\kern\dimen@\macc@nucleus\rel@kern{0.4}\kern\dimen@}%
      \advance\dimen@0.4\dimexpr\macc@kerna
      \let\final@kern#2%
      \ifdim\dimen@<\z@ \let\final@kern1\fi
      \if\final@kern1 \kern-\dimen@\fi
    \else
      \overline{\rel@kern{-0.6}\kern\dimen@#1}%
    \fi
  }%
  \macc@depth\@ne
  \let\math@bgroup\@empty \let\math@egroup\macc@set@skewchar
  \mathsurround\z@ \frozen@everymath{\mathgroup\macc@group\relax}%
  \macc@set@skewchar\relax
  \let\mathaccentV\macc@nested@a
  \if#31
    \macc@nested@a\relax111{#1}%
  \else
    \def\gobble@till@marker##1\endmarker{}%
    \futurelet\first@char\gobble@till@marker#1\endmarker
    \ifcat\noexpand\first@char A\else
      \def\first@char{}%
    \fi
    \macc@nested@a\relax111{\first@char}%
  \fi
  \endgroup
}
\makeatother

\numberwithin{equation}{section}

\theoremstyle{plain}
\newtheorem{theorem}{Theorem}[section]
\newtheorem{corollary}[theorem]{Corollary}
\newtheorem{proposition}[theorem]{Proposition}

\theoremstyle{definition}
\newtheorem{definition}[theorem]{Definition}
\newtheorem{remark}[theorem]{Remark}

\newcommand{\R}{\mathbb{R}}
\newcommand{\rn}{\R^n}

\newcommand{\N}{\mathbb{N}}
\newcommand{\M}{\mathscr{M}}
\newcommand{\Mpl}{\M^+}
\newcommand{\RR}{R}

\newcommand{\RM}{(\RR,\mu)}

\newcommand{\MRM}{\M\RM}
\newcommand{\MOR}{\M_0\RM}

\newcommand{\fund}{\varphi}
\newcommand{\fundMin}{\fund^{min}}
\newcommand{\fundMax}{\fund^{max}}

\newcommand*\dd{\mathop{}\!\mathrm{d}}

\DeclareMathOperator{\spt}{supp}

\DeclareMathOperator*{\ac}{\overset{*}{\hookrightarrow}}
\DeclareMathOperator*{\acLoc}{\overset{*}{\underset{loc}{\hookrightarrow}}}
\DeclareMathOperator*{\acInfty}{\overset{*}{\underset{\infty}{\hookrightarrow}}}
\DeclareMathOperator*{\weakacInfty}{\underset{\infty}{\hookrightarrow}}

\title{\TITLE}
\author{Zden\v ek Mihula}
\address{Zden\v ek Mihula, Czech Technical University in Prague, Faculty of Electrical Engineering, Department of Mathematics, Technick\'a~2, 166~27 Praha~6, Czech Republic}
\email{mihulzde@fel.cvut.cz}
\urladdr{\href{https://orcid.org/0000-0001-6962-7635}{0000-0001-6962-7635}}

\author{Maximili\'an P\'andy}
\address{Maximili\'an P\'andy, Department of Mathematical Analysis,
Faculty of Mathematics and Physics,
Charles University,
So\-ko\-lo\-vsk\'a~83,
186~75 Praha~8,
Czech Republic; School of Computation, Information and Technology, Technical University of Munich, Arcisstrasse 21, 80333 Munich, Germany}
\email{maxpandy123@gmail.com}

\begin{document}
\setcitestyle{numbers}
\bibliographystyle{plainnat}

\thanks{This research was partly supported by grant no.~23-04720S of the Czech Science Foundation and by Danube Region Grant no.~8X25005 (SOFTA) of the Czech Ministry of Education, Youth and Sports}

\begin{abstract}
We introduce and study two new relations between function spaces over measure spaces of infinite measure, motivated by the question of establishing compactness. The first relation captures the uniform decay of function (quasi-)norms ``at infinity''. It appeared implicitly in the first author's recent work on the compactness of Sobolev embeddings of radially symmetric functions on $\rn$. The second is a suitably localized version of the relation of almost-compact embeddings, which has been successfully used to study compactness in function spaces over measure spaces of finite measure, but becomes of no use in the case of infinite measure.  Our framework is that of quasi-Banach function spaces, which need not be normable or rearrangement invariant. This level of generality leads us to introduce the notion of  extremal fundamental functions associated with a (quasi-)Banach function space. We provide several concrete examples and establish an abstract compactness principle involving the new relations. Finally, we demonstrate a possible application of this principle to embeddings of inhomogeneous Sobolev spaces on $\rn$.
\end{abstract}

\maketitle


\section{Introduction}
This paper is devoted to studying two new relations between function spaces, motivated by the notorious problem of establishing compactness in function spaces over measure spaces of infinite measure. Before introducing these relations, we first explain their origin.

Let $T\colon A \to Y$ be a bounded (linear) operator from a normed space $A$ to a Banach function space $Y$ over a nonatomic $\sigma$-finite measure space $\RM$. An often applicable general approach to establishing compactness of $T$ can be divided into two independent steps. The first is to show that $T$ maps every bounded sequence in $A$ to a sequence in $Y$ that contains a subsequence converging pointwise $\mu$-a.e. This requires a close interplay between the structure of $T$ and that of $A$, and typically must be addressed on a case-by-case basis. In particular, this step cannot be formulated purely in terms of function spaces without knowledge of the operator $T$ itself. The second step is to show that $T$ maps bounded sets in $A$ to sets with \emph{uniformly absolutely continuous norm} in $Y$. Recall that $M\subseteq Y$ has uniformly absolutely continuous norm if
\begin{equation*}
    \lim_{n\to\infty} \sup_{f\in M} \|f\chi_{E_n}\|_Y = 0
\end{equation*}
for every sequence $\{E_n\}_{n=1}^\infty\subseteq\RR$ of $\mu$-measurable sets such that $\chi_{E_n} \to 0$ pointwise $\mu$-a.e. Assuming that both steps can be carried out, the compactness of $T$ follows (see, e.g.,~\cite[pages~31--32]{BS} or \cite{LZ:63}). To show that $T$ maps bounded sets in $A$ to sets with uniformly absolutely continuous norm in $Y$, one can often proceed as follows. One finds a Banach function space $Z$ such that $T\colon A \to Z$ is bounded and such that every bounded set in $Z$ has uniformly absolutely continuous norm in $Y$. We denote this property by $Z\ac Y$. Once this is verified, the conclusion follows easily. For example, when $0<\mu(\RR)<\infty$,
\begin{equation*}
    L^p\RM \ac L^q\RM \qquad \text{if and only if} \qquad p > q.
\end{equation*}
Note that $Z\ac Y$ is a relation between two function spaces that does not involve the operator $T$ at all. This approach, which implicitly or explicitly uses the relation $\ac$, has been successfully used to establish (and in some cases characterize) the compactness of various Sobolev (trace) embeddings and integral operators (see~\cite{CM:19,CM:22,EGN:23,F:18,KP:08,LZ:63, MPT:25,P:06,S:15}). The general theory of the relation $\ac$, called \emph{almost-compact embeddings}, between two Banach function spaces was studied in \cite{S:12} (cf.~\cite{F-MMP:10}).

Although the relation $\ac$ has proved to be very useful in questions involving compactness, there is an important catch in the approach described in the preceding paragraph. The issue is that there is no pair of Banach function spaces $Z$ and $Y$ such that $Z\ac Y$ when $\mu(\RR) = \infty$ (see~\cite[Theorem~4.5]{S:12}), which renders the approach futile when the underlying measure space has infinite measure. This is consistent with the general fact that compactness is much more delicate in function spaces over measure spaces of infinite measure. One of the obstacles absent when $\mu(\RR) < \infty$ is that ``mass can escape to infinity''\textemdash in more mathematical terms, pointwise a.e. convergence no longer implies convergence in measure. This phenomenon is the root of why the relation $\ac$ ceases to be useful when $\mu(\RR) = \infty$ (see~\cite[Theorem~3.1]{S:12}). An important example of a situation where we usually have compactness at our disposal when $\mu(\RR) < \infty$, but lose it when $\mu(\RR) = \infty$, is provided by Sobolev embeddings. For instance, the Sobolev embedding
\begin{equation}\label{intro:Sob_dom}
    W^{m,p}(\Omega) \hookrightarrow L^q(\Omega),
\end{equation}
 where $1\leq p < n/m$ and $\Omega\subseteq\rn$ is a sufficiently regular bounded domain, is compact if (and only if) $q < np/(n-mp)$. However, the corresponding embedding on the whole space,
 \begin{equation}\label{intro:Sob_rn}
     W^{m,p}(\rn) \hookrightarrow L^q(\rn),
 \end{equation}
is never compact. Nevertheless, compactness arguments remain useful, but they require more subtle analysis and additional considerations (see~\cite{L:84,L:84b,L:85}). For example, it is known (see~\cite{L:82,S:77}) that the restricted Sobolev embedding
\begin{equation}\label{intro:Sob_rad_rn}
    W^{m,p}_R(\rn) \hookrightarrow L^q(\rn),
\end{equation}
where $W^{m,p}_R(\rn)$ denotes the subspace of radially symmetric functions in $W^{m,p}(\rn)$, is compact if (and only if) $p < q < np/(n-mp)$. From the point of view of the procedure described in the preceding paragraph, the key difference between \eqref{intro:Sob_dom} and \eqref{intro:Sob_rn} (or \eqref{intro:Sob_rad_rn}) lies in the application of the relation $\ac$ at the end of the second step. A bounded sequence in $W^{m,p}$, regardless of whether the underlying domain has finite or infinite measure, always contains a subsequence converging pointwise~a.e. This follows from the classical Rellich-Kondrachov theorem. As for the space $Z$ used in the second step, one may in both cases take $Z = L^{np/(n-mp)}$\textemdash even better, one can replace the Lebesgue space $L^{np/(n-mp)}$ with the smaller Lorentz space $L^{np/(n-mp),p}$\textemdash but $L^{np/(n-mp)}(\Omega) \ac L^q(\Omega)$ holds for $q < np/(n-mp)$ only when the measure of $\Omega$ is finite. It is worth noting that replacing the Lebesgue space $L^{np/(n-mp)}$ with the strictly smaller Lorentz space $L^{np/(n-mp),p}$, or even with $L^p\cap L^{np/(n-mp),p}$, does not remedy the situation. Furthermore, the presence of compactness in \eqref{intro:Sob_rad_rn} and its absence in the unrestricted Sobolev embedding \eqref{intro:Sob_rn} indicate that one cannot hope to replace the relation $\ac$ with a single relation that would serve a similar role for function spaces over measure spaces of infinite measure.

The compactness of Sobolev embeddings on $\rn$ restricted to radially symmetric functions, in the general setting of rearrangement-invariant Banach function spaces, was recently studied and characterized in \cite{M:25preprint}. The characterization obtained there implicitly involves two new relations between function spaces, which we denote by $\acLoc$ and $\acInfty$. The aim of this paper is to provide the first systematic study of these relations. Moreover, we investigate them in the general framework of quasi-Banach function spaces (see~Section~\ref{sec:prel} for precise definitions), which need not be normable in general. Furthermore, the fact that we do not limit ourselves to rearrangement-invariant spaces leads us to introduce a new concept of extremal fundamental functions of (quasi-)Banach function spaces (see~Section~\ref{sec:extremal_fund_funcs}). While the classical concept of the fundamental function of a rearrangement-invariant space is well known and widely used, the corresponding notion of extremal fundamental functions, agreeing with the classical one in the rearrangement-invariant setting, does not appear in the existing literature and seems to be new, to the best of our knowledge.

The relation $\acLoc$ is a localized version of the relation $\ac$ studied in \cite{S:12}. Its definition is:
\begin{definition}
Let $X$ and $Y$ be quasi-Banach function spaces over $\RM$. We write $X\acLoc Y$ if
\begin{equation*}
    \lim_{a\to0_+}\sup_{\|f\|_X\leq 1} \sup_{\substack{E\subseteq\RR\\ \mu(E)\leq a}} \|f\chi_{E}\|_Y = 0.
\end{equation*}
\end{definition}
When $\mu(\RR)<\infty$, the relation $\acLoc$ coincides with $\ac$ (see~\cite[Lemma~5.1]{S:12}), but the situation changes dramatically when $\mu(\RR) = \infty$. For example,
\begin{align*}
    L^p\RM &\ac L^q\RM \quad \text{never holds when $\mu(\RR) = \infty$},\nonumber\\
    \intertext{but}
    L^p\RM &\acLoc L^q\RM \quad \text{holds if and only if $p > q$}.
\end{align*}
The key difference between $\ac$ and $\acLoc$ can be seen by comparing their characterizations in terms of convergence of functions \cite[Theorem~3.1]{S:12} and Theorem~\ref{thm:characterization_loc_ac_convergence}, respectively. Whereas the characterization of $\ac$ involves pointwise a.e.~convergence of functions (equivalently, it may be reformulated by means of local convergence in measure), our characterization of $\acLoc$ involves convergence of the measure of supports.

From now on, we assume that $\mu(\RR) = \infty$. The relation $\acInfty$ is defined as follows.
\begin{definition}
Let $X$ and $Y$ be quasi-Banach function spaces over $\RM$. We write $X\acInfty Y$ if
\begin{equation*}
    \lim_{a\to\infty}\sup_{\|f\|_X\leq 1} \inf_{\substack{E\subseteq\RR\\ \mu(E)\leq a}} \|f\chi_{\RR\setminus E}\|_Y = 0.
\end{equation*}
\end{definition}
The study of the relation $\acInfty$ is the main focus of this paper, as it is entirely new, to the best of our knowledge. For instance, to give the reader some intuition,
\begin{equation*}
    L^p\RM \acInfty L^q\RM \quad \text{holds if and only if $p < q$}.
\end{equation*}
The relations $\acInfty$ and $\acLoc$ are investigated in Section~\ref{sec:uniform_decay}. In Section~\ref{sec:endpoints}, we restrict our attention to rearrangement-invariant Banach function spaces and study the relation $\acInfty$ in connection with the so-called endpoint spaces, which play an important role in the theory of rearrangement-invariant spaces. Section~\ref{sec:examples} contains several concrete examples and applications. In particular, we characterize there the validity of $\acInfty$ for pairs of Lorentz spaces and pairs of Orlicz spaces. In addition, we prove a general compactness principle (see~Theorem~\ref{thm:Lions_compactness_lemma}), which may be viewed as an abstract variant of Lions' compactness lemma \cite[Lemma~I.1]{L:84b}, illustrating how the relations $\acLoc$ and $\acInfty$ together provide a tool for establishing compactness. Finally, we apply the general compactness principle to (inhomogeneous) Sobolev spaces on $\rn$ to demonstrate a concrete application (see~Theorem~\ref{thm:application_Lions_compactness_inhom_Sob}).


\section{Preliminaries}\label{sec:prel}
The theory of (rearrangement-invariant) Banach function spaces used in this paper follows \cite{BS}. As noted in \cite{LN:24}, many results can still be proved under milder assumptions. Nonetheless, we follow \cite{BS} to avoid unnecessary complications in establishing the fundamental theory of the new relation $\acInfty$, as well as of $\acLoc$, introduced here. While the classical theory of (rearrangement-invariant) Banach function spaces is well developed, the corresponding theory for (rearrangement-invariant) quasi-Banach function spaces is comparatively less explored. Recent works \cite{MNPT:25, NP:24} extended key results from \cite{BS} to the quasi-normable setting, and we follow their approach here.

Throughout the paper $\RM$ denotes a $\sigma$-finite, nonatomic measure space with $\mu(\RR)=\infty$, unless stated otherwise. We denote by $\MRM$ the set of $\mu$-measurable functions on $\RR$ and by $\Mpl\RM$ its subset consisting of those functions that are nonnegative $\mu$-a.e. We also denote by $\MOR$ the set of functions from $\MRM$ that are finite $\mu$-a.e.

Recall that the nonatomicity of $\RM$ ensures the following property (e.g., see~\cite[Proposition~215D]{Fremlin}): For every $t\in[0, \mu(\RR)]$, there exists a $\mu$-measurable set $E\subseteq\RR$ such that $\mu(E) = t$. We will frequently use this property, which is sometimes called the Sierpiński theorem for nonatomic measures, implicitly in the paper.

A functional $\|\cdot\|_X\colon  \Mpl\RM \to [0,\infty]$ is called a \emph{quasi-Banach function norm} if, for all $f$, $g$, and $\{f_j\}_{j=1}^\infty$ in $\Mpl\RM$, the following properties hold:
\begin{enumerate}[(P1)]
\item $\|f\|_X = 0$ if and only if $f = 0$ $\mu$-a.e.; 
      $\|\lambda f\|_X = \lambda \|f\|_X$ for every $\lambda \ge 0$; there is a constant $C_X \in [1,\infty)$, independent of $f$ and $g$, such that
      \begin{equation}\label{prel:P1_quasi_triangle}
          \|f + g\|_X \le C_X\big( \|f\|_X + \|g\|_X \big);
      \end{equation}
\item $f \leq g$ $\mu$-a.e.\  implies $\|f\|_X\leq\|g\|_X$;
\item $f_j \nearrow f$ $\mu$-a.e.\ implies $\|f_j\|_X \nearrow \|f\|_X$;
\item $\|\chi_E\|_X < \infty$ for every $E\subseteq R$ with $\mu(E) < \infty$.
\end{enumerate}
Given a quasi-Banach function norm $\|\cdot\|_X$, we extend it to all functions from $\MRM$ by setting
\begin{equation*}
    \|f\|_X = \||f|\|_{X},\ f\in \MRM.
\end{equation*}
Then $\|\cdot\|_X$ is a quasi-norm on the linear set
\begin{equation*}
X=\{f\in\MRM: \|f\|_X<\infty\},
\end{equation*}
and $X$ endowed with $\|\cdot\|_X$ is a quasi-Banach space, which we call a \emph{quasi-Banach function space}. We will often write \emph{q-BFS} for short. In what follows, we assume that $C_X$ is the least constant for which \eqref{prel:P1_quasi_triangle} holds. Therefore, $\|\cdot\|_X$ is a norm on $X$ if and only if $C_X = 1$. Every q-BFS $X$ satisfies
\begin{equation*}
    \{f\in\MOR\colon \text{$f$ is a simple function}\} \subseteq X \subseteq \MOR,
\end{equation*}
and the second inclusion is continuous when the space $\MOR$ is equipped with the (metrizable) topology of local convergence in measure. A simple function is a (finite) linear combination of characteristic functions of $\mu$-measurable sets of finite measure.

We say that a quasi-Banach function norm $\|\cdot\|_X\colon  \Mpl\RM \to [0,\infty]$ is a \emph{Banach function norm} if \eqref{prel:P1_quasi_triangle} holds with $C_X=1$ and, in addition, it satisfies
\begin{enumerate}[(P5)]
    \item for every set $E\subseteq R$ with $\mu(E) < \infty$, there exists a finite constant $C_{E,X}>0$ such that
		\begin{equation*}
			\int_E f \dd{\mu} \leq C_{E,X} \|f\|_X \quad \text{for every $f\in \Mpl\RM$}.
		\end{equation*}
\end{enumerate}
In particular, $\|\cdot\|_X$ is a norm on $X$, and the resulting Banach space $X$ is called a \emph{Banach function space}. We will often write \emph{BFS} for short. Note that a quasi-Banach function norm need not satisfy (P5), whereas (P5) is part of the definition of a Banach function norm.

With any quasi-Banach function norm $\|\cdot\|_X$, we associate another functional, $\|\cdot\|_{X'}$, defined by
\begin{equation}\label{prel:asoc_norm}
\|g\|_{X'}=\sup_{\|f\|_X\leq1}\int_{\RR} |f(x)||g(x)| \dd{\mu(x)},\ g\in\MRM.
\end{equation}
An immediate consequence of the definition of $\|\cdot\|_{X'}$ is the H\"older-type inequality
\begin{equation}\label{prel:Holder}
    \int_{\RR} |f(x)||g(x)| \dd{\mu(x)} \leq \|f\|_X \|g\|_{X'} \quad \text{for all $f,g\in\MRM$}.
\end{equation}
When $\|\cdot\|_X$ does not satisfy property~(P5), there exists a set
$E \subseteq \RR$ with $0 < \mu(E) < \infty$ such that $\|\chi_E\|_{X'} = \infty$.
Consequently, $\|\cdot\|_{X'}$ does not have property~(P4). In fact, it may happen that $\|g\|_{X'} = \infty$ for every $g\in\Mpl\setminus\{0\}$. For example,
this occurs for the Lebesgue space $X = L^p$ with $p \in (0,1)$ or the Lorentz
space $L^{1,q}$ with $q \in (1,\infty]$ (see Section~\ref{sec:examples} for more information on Lorentz spaces). On the other hand, when a quasi-Banach function norm $\|\cdot\|_X$ satisfies property (P5), the functional $\|\cdot\|_{X'}$ is a Banach function norm (see~\cite[Remark~2.3]{EKP:00} or \cite[Theorem~3.1]{GS:14}), and the resulting Banach function space $X'$ is called the \emph{associate space} of $X$. In turn, the second associate space $X'' = (X')'$ is also a Banach function space, and we always have
\begin{equation}\label{prel:X_into_second_asoc}
    \|f\|_{X''} \leq \|f\|_{X} \quad \text{for every $f\in\MRM$}.
\end{equation}
When $\|\cdot\|_X$ is a Banach function norm, \eqref{prel:X_into_second_asoc} becomes an equality. In other words, if $X$ is a BFS, then
\begin{equation}\label{prel:X_is_double_assoc}
    X'' = X \quad \text{with equal norms}.
\end{equation}

An important subclass of (quasi-)Banach function spaces is formed by the
rearrangement-invariant ones. If $(R,\mu)$ and $(S,\nu)$ are two (possibly different) measure spaces, we say that functions $f\in \M(R,\mu)$ and $g\in\M(S,\nu)$ are \textit{equimeasurable} if
\begin{equation*}
    \mu\big( \{x\in\RR: |f(x)| > \lambda\} \big) = \nu\big( \{x\in S: |g(x)| > \lambda\} \big) \quad \text{for every $\lambda>0$}.
\end{equation*}
We say that a quasi-Banach function norm $\|\cdot\|_X$ is \emph{rearrangement invariant} if it satisfies
\begin{enumerate}[(P6)]
\item $\|f\|_X = \|g\|_X$ whenever $f,g\in\Mpl\RM$ are equimeasurable.
\end{enumerate}
The resulting (quasi-)Banach function space is called a \emph{rearrangement-invariant (quasi-)Banach function space}. We will often write \emph{r.i.~q-BFS} or \emph{r.i.~BFS} for short.

Note that, to verify that an r.i.~q-BFS is an r.i.~BFS, it suffices to check that $\|\cdot\|_X$ is a norm, that is, \eqref{prel:P1_quasi_triangle} holds with $C_X = 1$. The fact that $X$ also satisfies property~(P5) then follows (see, e.g.,~\cite[Chapter~2, Theorem~4.1]{KPS} or \cite[Proposition~2.22]{MNPT:25}). 

If $X$ is an r.i.~q-BFS that satisfies property~(P5), then its associate space is an r.i.~BFS (see~\cite[Remark~2.3]{EKP:00}). In particular, the associate space of an r.i.~BFS is also an r.i.~BFS (see~\cite[Chapter~2, Proposition~4.2]{BS}).

Textbook examples of r.i.~BFSs include
Lebesgue spaces $L^{p}$ with $p\in[1,\infty]$, Orlicz spaces, and Lorentz
spaces $L^{p,q}$ for suitable values of $p,q\in[1,\infty]$ (see
Section~\ref{sec:examples} for more information). The Lorentz spaces $L^{1,q}$
with $q\in(0,1)$ are examples of rearrangement-invariant (non-normable)
quasi-Banach function spaces that satisfy property~(P5). In contrast, the
Lebesgue spaces $L^p$ with $p\in(0,1)$ and the Lorentz spaces $L^{1,q}$
with $q\in(1,\infty]$ are rearrangement-invariant (non-normable) quasi-Banach
function spaces that do not satisfy property~(P5). Variable exponent Lebesgue
spaces $L^{p(\cdot)}$ and suitable weighted Lebesgue spaces provide examples of
(quasi-)Banach function spaces that are not rearrangement invariant (see, e.g.,~\cite{DHHR} or Remark~\ref{rem:qBFS_that_is_not_admissible}).

The \emph{nonincreasing rearrangement} $f^* \colon  [0,\infty) \to [0, \infty]$ of a function $f\in \M(R,\mu)$  is defined by
\begin{equation*}
 f^*(t) = \inf\big\{ \lambda\geq0\colon \mu(\{x\in \RR: |f(x)|>\lambda\})\leq t \big\},\ t\in[0,\infty).
\end{equation*}
The function $f^*$ is nonincreasing, right-continuous, and equimeasurable with $f$. The \emph{maximal nonincreasing rearrangement} $f^{**} \colon  (0,\infty) \to [0, \infty]$ of a function $f\in \M(R,\mu)$ is defined by
\begin{equation*}
 f^{**}(t) = \frac1{t}\int_0^t f^*(s) \dd{s},\ t \in (0,\infty).
\end{equation*}
The function $f^{**}$ is nonincreasing and continuous. Unlike $f^*$, it need not be equimeasurable with $f$. We always have $f^*\leq f^{**}$ and $f^*(0) = f^{**}(0_+) = \|f\|_{L^\infty\RM}$. Moreover, $f^{**}$ admits the representation
\begin{equation*}
    f^{**}(t)  = \frac1{t} \sup_{\substack{E\subseteq \RR\\\mu(E) = t}} \int_E |f| \dd{\mu},\ t \in (0,\infty).
\end{equation*}

For every r.i.~q-BFS $X$ over $(\RR,\mu)$, there exists a unique r.i.~q-BFS $\widebar{X}$ over the interval $(0,\infty)$ endowed with the one-dimensional Lebesgue measure such that
\begin{equation*}
    \|f\|_X = \|f^*\|_{\widebar{X}} \quad \text{for every $f\in\MRM$}.
\end{equation*}
The space $\widebar{X}$ is called the \emph{representation space} of $X$.
Moreover, $C_X = 1$ if and only if $C_{\widebar{X}} = 1$, and $X$ satisfies
property~(P5) if and only if $\widebar{X}$ does. In particular, the
representation space of an r.i.~BFS is an r.i.~BFS. In the setting of r.i.~BFSs, this is a
classical result (see, e.g.,~\cite[Chapter~2, Theorem~4.10]{BS}). Recently, it was extended to the more general setting of r.i.~q-BFSs in \cite[Section~3]{MNPT:25}. When $\RR=(0,\infty)$ and $\mu$ is the Lebesgue measure, $X$ coincides with its representation space, with equal norms.

An important function associated with each r.i.~q-BFS $X$ is its \emph{fundamental function} $\fund_X\colon[0, \infty) \to [0, \infty)$,  defined by
\begin{equation*}
    \fund_X(t) = \|\chi_E\|_{X},\ t\in[0,\infty),
\end{equation*}
where $E\subseteq \RR$ is any $\mu$-measurable set with $\mu(E) = t$. The rearrangement invariance of $X$ ensures that $\fund_X$ is well defined, and it can be readily verified that $\fund_X = \fund_{\widebar{X}}$. Furthermore, the elementary pointwise estimate
\begin{equation}\label{prel:fund_ineq}
    f^*(t) \leq \frac{\|f\|_X}{\fund_X(t)} \quad \text{for all $t\in(0, \infty)$ and $f\in\MRM$}
\end{equation}
is often useful (cf.~\cite{H}).

The fundamental function of an r.i.~BFS is a \emph{quasiconcave} function. A function $\varphi\colon [0, \infty) \to [0, \infty)$ is called quasiconcave if it is nondecreasing, vanishes only at $0$, and the function $(0, \infty) \ni t \mapsto \varphi(t)/t$ is nonincreasing.

\section{Extremal fundamental functions}\label{sec:extremal_fund_funcs}
A general (quasi-)Banach function space does not possess a single, well-defined fundamental function, because the norm of characteristic functions of sets with prescribed measure may vary. Nevertheless, one can still define two \emph{extremal fundamental functions}
serving as a certain substitute. Even when $X$ is not rearrangement invariant, these
extremal fundamental functions, introduced below, retain some useful
information about the space, as we will see later. To the best of our
knowledge, this notion does not appear in the existing literature.

\begin{definition}
    Let $X$ be a quasi-Banach function space. The function $\fundMin_X\colon [0, \infty) \to [0, \infty)$ defined by
    \begin{equation*}
        \fundMin_X(t) = \inf_{\mu(E) = t} \|\chi_E\|_X,\ t\in[0, \infty),
    \end{equation*}
    is called the \emph{minimal fundamental function} of $X$. The function $\fundMax_X\colon [0, \infty) \to [0, \infty)$ defined by
    \begin{equation*}
        \fundMax_X(t) = \sup_{\mu(E) = t} \|\chi_E\|_X,\ t\in[0, \infty),
    \end{equation*}
    is called the \emph{maximal fundamental function} of $X$. Finally, we say that $X$ is \emph{admissible} if the minimal fundamental function vanishes only at the origin, that is,
    \begin{equation*}
        \fundMin_X(t) > 0 \quad \text{for every $t\in(0,\infty)$}.
    \end{equation*}
\end{definition}
The functions $\fundMin_X$ and $\fundMax_X$ are nondecreasing, finite, and vanishing at the origin. Moreover, we clearly always have $\fundMin_X\leq\fundMax_X$. Their finiteness is immediate when $X$ is rearrangement invariant, since in this case
\begin{equation*}
	\fundMin_X(t) = \fundMax_X(t) = \fund_X(t) < \infty \quad \text{for every $t \in [0,\infty)$}
\end{equation*}
by property~(P4). Furthermore, every r.i.~q-BFS $X$ is admissible, since $\fundMin_X(t) = \fundMax_X(t) = \fund_X(t) > 0$ for every $t\in(0, \infty)$. For general (quasi-)Banach function spaces that are not rearrangement invariant, the finiteness of extremal fundamental functions is less obvious. Nevertheless, it still holds. For Banach function spaces, the finiteness of $\fundMax_X$ (and therefore also of $\fundMin_X$) follows from \cite[Lemma~4.4]{S:12}, and the argument there can be adapted to the quasi-Banach setting. Moreover, \cite[Lemma~4.4]{S:12} also implies that every Banach function space is admissible.  Finally, since every BFS is admissible, we observe that every q-BFS satisfying property (P5) is also admissible.

\begin{proposition}\label{prop:quasi_BFS_with_P5_admissible}
    Every q-BFS that satisfies property~(P5) is admissible. 
\end{proposition}
\begin{proof}
We need to show that
    \begin{equation*}
        \inf_{\mu(E) = t} \|\chi_E\|_X = \fundMin_X(t) > 0 \quad \text{for every $t\in(0, \infty)$}.
    \end{equation*}
    Since $X$ satisfies property~(P5), its associate space $X'$ is a BFS; in turn, so is the second associate space $X''$. Since BFSs are admissible, it follows that $X''$ is admissible. Hence, using the admissibility of $X''$ and \eqref{prel:X_into_second_asoc}, we obtain
    \begin{equation*}
        \fundMin_{X}(t) \geq \fundMin_{X''}(t) > 0 \quad \text{for every $t\in(0, \infty)$}. \qedhere
    \end{equation*}
\end{proof}

\begin{remark}\label{rem:qBFS_that_is_not_admissible}
    When $X$ is a q-BFS that does not satisfy property~(P5), it may happen that
    \begin{equation*}
        \fundMin_X(t) = 0 \quad \text{for every $t\in(0, \infty)$}.
    \end{equation*}
    For example, let $X=L^1_w(0, \infty)$ be the weighted Lebesgue space corresponding to the function norm
    \begin{equation*}
        \|f\|_X = \int_0^\infty f(t)\min\{1, t^{-2}\} \dd{t},\ f\in\Mpl(0, \infty).
    \end{equation*}
    It can be readily verified that the functional $\|\cdot\|_X$ satisfies all properties of a Banach function norm except for property~(P5). In particular, $X$ is a q-BFS. However, for all $n\in\N$ and $t\in(0, \infty)$, we have
    \begin{equation*}
        \fundMin_X(t) \leq \|\chi_{(n,n+t)}\|_X = \frac1{n} - \frac1{n+t} \rightarrow0 \quad \text{as $n\rightarrow\infty$}.
    \end{equation*}
    Thus $\fundMin_X(t) = 0$ for every $t\in(0, \infty)$. On the other hand, there exist q-BFSs that do not satisfy property~(P5) but are admissible nonetheless, since every r.i.~q-BFS is admissible. For concrete examples, consider the Lorentz spaces $L^{1,q}$ with $q \in (1,\infty]$ or the Lebesgue spaces $L^p$ with $p \in (0,1)$.
\end{remark}

The fundamental functions of an r.i.~BFS $X$ and its associate space $X'$ satisfy
\begin{equation}\label{prel:fund_X_and_asocX_for_ri}
    \fund_X(t) \fund_{X'}(t) = t \quad \text{for every $t\in[0, \infty)$},
\end{equation}
a relation that plays an important role in the theory of rearrangement-invariant Banach function spaces. The next proposition provides a substitute for this identity in the setting of general (quasi-)Banach function spaces. In order to ensure that $X'$ is nontrivial, we assume that $X$ satisfies property~(P5). Before stating the result, we introduce an auxiliary notion. For each $E\subseteq\RR$ such that $0<\mu(E)<\infty$, we define the averaging operator $A_E\colon L^1_{loc}\RM \to (L^1\cap L^\infty)\RM$ by
\begin{equation*}
    A_Ef = \Big( \frac1{\mu(E)}\int_E f \dd{\mu} \Big) \chi_E,\ f\in L^1_{loc}\RM.
\end{equation*}
Now, let $X$ be a q-BFS that satisfies property~(P5). By the H\"{o}lder inequality \eqref{prel:Holder} and the nontriviality of $X'$, each $A_E$ is bounded on $X$.  We say that \emph{the family of simple averaging operators $\{A_E\}_E$ is uniformly bounded on $X$} if
\begin{equation*}
    \sup_{\substack{E\subseteq\RR\\0<\mu(E) < \infty}} \sup_{\|f\|_X\leq1} \|A_Ef\|_{X} < \infty.
\end{equation*}
It follows from the uniform boundedness principle that the family of simple averaging operators $\{A_E\}_E$ is uniformly bounded on $X$ if and only if
\begin{equation*}
    \sup_{\substack{E\subseteq\RR\\0<\mu(E) < \infty}} \|A_Ef\|_{X} < \infty \quad \text{for every $f\in X$ with $\|f\|_X\leq1$}.
\end{equation*}

\begin{proposition}\label{prop:E:fund_X_and_X'_quasi}
    Let $X$ be a q-BFS that satisfies property~(P5).
\begin{enumerate}[(i)]
\item We have
\begin{equation}\label{E:fund_X_and_X'_quasi:at_least_identity}
        t\leq \min\big\{ \fundMin_X(t)\fundMax_{X'}(t), \fundMax_X(t)\fundMin_{X'}(t) \big\}  \quad\text{for every $t\in(0, \infty)$}.
\end{equation}
\item Assume that the family of simple averaging operators $\{A_E\}_E$ is uniformly bounded on $X$. Then there exists a constant $C>0$ such that
\begin{equation}\label{E:fund_X_and_X'_quasi:at_most_identity}
    \max\big\{ \fundMin_X(t)\fundMax_{X'}(t), \fundMax_X(t)\fundMin_{X'}(t) \big\} \leq C t \quad \text{for every $t\in(0, \infty)$}.
\end{equation}
\end{enumerate}
\end{proposition}
\begin{proof}
Since $X$ satisfies property~(P5), both $X$ and $X'$ are admissible by Proposition~\ref{prop:quasi_BFS_with_P5_admissible} and the fact that $X'$ is a BFS. The inequality \eqref{E:fund_X_and_X'_quasi:at_least_identity} follows directly from the H\"{o}lder inequality \eqref{prel:Holder}. Indeed, fixing $t\in(0, \infty)$ and $E\subseteq\RR$ with $\mu(E) = t$, we obtain
    \begin{equation*}
        t = \int_{\RR} \chi_E \dd{\mu} \leq \|\chi_E\|_X \|\chi_E\|_{X'} \leq \min\{\|\chi_E\|_X \fundMax_{X'}(t), \fundMax_X(t) \|\chi_E\|_{X'} \},
    \end{equation*}
    and taking the infimum over all such $E$ yields \eqref{E:fund_X_and_X'_quasi:at_least_identity}. For \eqref{E:fund_X_and_X'_quasi:at_most_identity}, assume that the family of simple averaging operators $\{A_E\}_E$ is uniformly bounded on $X$. Thus there exists $C>0$ such that
    \begin{equation*}
        C = \sup_{\substack{E\subseteq\RR\\0<\mu(E) < \infty}} \sup_{\|f\|_X\leq1} \|A_Ef\|_{X} < \infty.
    \end{equation*}
    Fix $t>0$ and $E\subseteq\RR$ with $\mu(E) = t$. Using \eqref{prel:asoc_norm}, we have
    \begin{align*}
        \|\chi_E\|_{X'} &= \sup_{\|f\|_{X}\leq1} \int_E |f| \dd{\mu} = \frac{t}{\fundMin_X(t)} \sup_{\|f\|_{X}\leq1} \Big( \frac1{t} \int_E |f| \dd{\mu} \Big) \fundMin_X(t) \\
        &\leq \frac{t}{\fundMin_X(t)} \sup_{\|f\|_{X}\leq1} \Big( \frac1{t} \int_E |f| \dd{\mu} \Big) \|\chi_E\|_X =  \frac{t}{\fundMin_X(t)} \sup_{\|f\|_{X}\leq1} \|A_E(|f|)\|_X \\
        &\leq C \frac{t}{\fundMin_X(t)}.
    \end{align*}
    It follows that
    \begin{equation*}
        \fundMin_X(t) \fundMax_{X'}(t) \leq Ct \quad \text{for every $t\in(0, \infty)$}.
    \end{equation*}
    Finally, fix an arbitrary $\varepsilon > 0$ and choose $E\subseteq\RR$ with $\mu(E) = t$ such that
    \begin{equation*}
        \fundMax_X(t) \leq (1+\varepsilon)\|\chi_E\|_X.
    \end{equation*}
    Using a similar argument, one obtains
    \begin{equation*}
        \fundMin_{X'}(t) \leq \|\chi_E\|_{X'} \leq (1+\varepsilon)C \frac{t}{\fundMax_X(t)}.
    \end{equation*}
    It follows that
    \begin{equation*}
        \fundMax_X(t) \fundMin_{X'}(t) \leq Ct \quad \text{for every $t\in(0, \infty)$},
    \end{equation*}
    which completes the proof.
\end{proof}

\begin{remark}\label{rem:simple_aver_unif_bd_on_ri}
    The identity \eqref{prel:fund_X_and_asocX_for_ri} together with the H\"{o}lder inequality \eqref{prel:Holder} immediately implies that the family of simple averaging operators $\{A_E\}_E$ is uniformly bounded on every r.i.~BFS $X$ with constant $1$ (cf.~\cite[Chapter~2, Theorem~4.8]{BS}). Therefore, Proposition~\ref{prop:E:fund_X_and_X'_quasi} recovers the identity \eqref{prel:fund_X_and_asocX_for_ri}.
\end{remark}

When $X$ is rearrangement invariant, its extremal fundamental functions coincide with the classical fundamental function, that is, $\fundMin_X = \fundMax_X = \fund_X$, and \eqref{E:fund_X_and_X'_quasi:at_most_identity} reads as
\begin{equation}\label{E:fund_X_and_X'_quasi:at_most_identity_for_ri}
        \fund_X(t) \fund_{X'}(t) \leq Ct \quad \text{for every $t\in(0, \infty)$}.
\end{equation}
When \eqref{E:fund_X_and_X'_quasi:at_most_identity_for_ri} holds, the H\"{o}lder inequality \eqref{prel:Holder} together with the fact that $\fundMin_X = \fundMax_X = \fund_X$ implies that the family of simple averaging operators is uniformly bounded on $X$. By combining this observation with Proposition~\ref{prop:E:fund_X_and_X'_quasi}, we obtain the following characterization of the validity of \eqref{E:fund_X_and_X'_quasi:at_most_identity_for_ri}. Interestingly, a different characterization was recently obtained in \cite[Corollary~4.22]{MNPT:25}.
\begin{corollary}\label{cor:E:fund_X_and_X'_quasi_ri}
        Let $X$ be an r.i.~q-BFS satisfying property~(P5). Then there exists a constant $C>0$ such that
        \begin{equation*}
            t\leq \fund_X(t)\fund_{X'}(t) \leq Ct \quad \text{for every $t\in(0, \infty)$}
        \end{equation*}
        if and only if the family of simple averaging operators $\{A_E\}_E$ is uniformly bounded on $X$.
    \end{corollary}

    \begin{remark}
        When $X$ is an r.i.~q-BFS satisfying property~(P5), a sufficient condition for the uniform boundedness of the family of simple averaging operators is the boundedness of the mapping $f \mapsto f^{**}$ on the representation space $\widebar{X}$. This condition is not necessary, since the family of simple averaging operators is uniformly bounded on every r.i.~BFS\textemdash for example, on $L^1$. 
        
        Furthermore, there exist r.i.~q-BFSs satisfying property~(P5) which are not normable and on which the family of simple averaging operators is uniformly bounded, while $f\mapsto f^{**}$ is not bounded. A simple example is the Lorentz space $X = L^{1,q}$ with $q \in (0,1)$. Then $\fund_X(t) = c_q t$, $t\in(0, \infty)$, $X' = L^\infty$, up to equivalent norms, and one can readily verify, using Corollary~\ref{cor:E:fund_X_and_X'_quasi_ri}, that the family of simple averaging operators is uniformly bounded on $X$.
    \end{remark}

\section{Uniform decay of functions}\label{sec:uniform_decay}
In this section, we study some fundamental properties of the relations $\acLoc$ and $\acInfty$. We start with a useful necessary condition for the validity of $X\acLoc Y$ or $X\acInfty Y$, involving the extremal fundamental functions.
\begin{proposition}\label{prop:necessaryCond_fundFuncs_decay}
Let $X$ and $Y$ be q-BFSs, with $X$ being admissible.
\begin{enumerate}[(i)]
    \item If $X\acInfty Y$, then
\begin{equation}\label{E:necessaryCond_fundFuncs_decay:infty}
    \lim_{a\to \infty} \frac{\fundMin_Y(a)}{\fundMin_X(2a)} = 0.
\end{equation}
In particular, we necessarily have $\lim_{a\to \infty} \fundMin_X(a) = \infty$.
\item If $X\acLoc Y$, then
\begin{equation}\label{E:necessaryCond_fundFuncs_decay:zero}
    \lim_{a\to 0_+} \frac{\fundMax_Y(a)}{\fundMax_X(a)} = 0.
\end{equation}
In particular, we necessarily have $\lim_{a\to 0_+} \fundMax_Y(a) = 0$.
\end{enumerate}
\end{proposition}
\begin{proof}
We start by proving \eqref{E:necessaryCond_fundFuncs_decay:infty}. Suppose that there exist $\varepsilon_0>0$ and a sequence $\{a_k\}_{k=1}^\infty\subseteq(0, 
    \infty)$ such that $\lim_{k\to\infty}a_k=\infty$ and
    \begin{equation}\label{E:necessaryCond_fundFuncs_decay:1}
        \frac{\fundMin_Y(a_k)}{\fundMin_X(2a_k)} \geq \varepsilon_0 \quad \text{for every $k\in\N$}.
        \end{equation}
        For each $k\in\N$, choose a set $F_k\subseteq\RR$ with $\mu(F_k) = 2a_k$ such that
        \begin{equation*}
            \fundMin_X(2a_k) \leq \|\chi_{F_k}\|_X \leq 2\fundMin_X(2a_k).
        \end{equation*}
        Define
        \begin{equation*}
            f_k = \frac{\chi_{F_k}}{2\fundMin_X(2a_k)},\ k\in\N.
        \end{equation*}
				Clearly, we have $\|f_k\|_X \leq 1$ for every $k\in\N$. Furthermore, note that for every $k\in\N$, $\mu(F_k \setminus E) \geq a_k$ for every $E\subseteq \RR$ with $\mu(E)\leq a_k$. Hence,
        \begin{equation*}
            \sup_{\|f\|_X\leq1} \inf_{\mu(E)\leq a_k}\|f\chi_{\RR\setminus E}\|_Y \geq \inf_{\mu(E)\leq a_k} \frac{\|\chi_{F_k\setminus E}\|_Y}{2\fundMin_X(2a_k)} \geq  \frac{\fundMin_Y(a_k)}{2\fundMin_X(2a_k)}
        \end{equation*}
        for every $k\in\N$. Combining this with \eqref{E:necessaryCond_fundFuncs_decay:1}, we conclude that $X\acInfty Y$ cannot hold.

        We now turn our attention to \eqref{E:necessaryCond_fundFuncs_decay:zero}, whose proof is similar to that of \eqref{E:necessaryCond_fundFuncs_decay:infty}. Suppose that there exist $\varepsilon_0>0$ and a sequence $\{a_k\}_{k=1}^\infty\subseteq(0, 
    \infty)$ such that $\lim_{k\to\infty}a_k=0$ and
    \begin{equation}\label{E:necessaryCond_fundFuncs_decay:2}
        \frac{\fundMax_Y(a_k)}{\fundMax_X(a_k)} \geq \varepsilon_0 \quad \text{for every $k\in\N$}.
        \end{equation}
        For each $k\in\N$, choose $E_k\subseteq\RR$ with $\mu(E_k) = a_k$ such that $\|\chi_{E_k}\|_Y \geq \fundMax_Y(a_k)/2$. Define
        \begin{equation*}
            f_k = \frac{\chi_{E_k}}{\fundMax_X(a_k)},\ k\in\N.
        \end{equation*}
        Clearly, we have $\|f_k\|_{X}\leq1$ for every $k\in\N$. Then
        \begin{equation*}
            \sup_{\|f\|_X\leq 1} \sup_{\mu(E)\leq a_k} \|f\chi_{E}\|_Y \geq \sup_{\mu(E)\leq a_k} \|f_k\chi_{E}\|_Y \geq \frac{\|\chi_{E_k}\|_Y}{\fundMax_X(a_k)} \geq \frac{\fundMax_Y(a_k)}{2\fundMax_X(a_k)}
        \end{equation*}
        for every $k\in\N$. Combining this with \eqref{E:necessaryCond_fundFuncs_decay:2}, we see that $X\acLoc Y$ does not hold.
\end{proof}

\begin{remark}\label{rem:necessaryCond_fundFuncs_decay_ri}
When $X$ is rearrangement invariant, we have $\fundMin_X=\fundMax_X =\fund_X$ and $\fund_X(2a)\leq 2C_X\fund_X(a)$ for every $a\in(0, \infty)$, since any set of measure $2a$ can be written as a disjoint union of two sets of measure $a$. Hence, in the rearrangement-invariant case, $\fundMin_X(2a)$ in 
\eqref{E:necessaryCond_fundFuncs_decay:infty} can be replaced by $\fund_X(a)$. 

Furthermore, it can be readily verified that $X$ need not be admissible for \eqref{E:necessaryCond_fundFuncs_decay:zero} to hold. In particular, \eqref{E:necessaryCond_fundFuncs_decay:zero} implies that there is no q-BFS $X$ such that $X \acLoc X$. 

Finally, when $X$ is admissible, a slight modification of the argument in the proof of 
\eqref{E:necessaryCond_fundFuncs_decay:zero} shows that the same conclusion holds with both maximal fundamental functions replaced by the respective minimal fundamental functions.
\end{remark}

We now characterize both relations $\acLoc$ and $\acInfty$ in terms of convergence of functions. Not only are these characterizations useful from applications' point of view, but they also reveal the essence of these relations. We start with $\acInfty$.
\begin{theorem}\label{thm:characterization_infty_convergence}
    Let $X$ and $Y$ be q-BFSs, with $X$ being admissible. The following two statements are equivalent.
    \begin{enumerate}[(i)]
        \item $X\acInfty Y$.
        \item We have $\lim_{t\to\infty}\fundMin_X(t) = \infty$ and for every bounded sequence $\{f_n\}_{n = 1}^\infty$ in $X$ such that $\lim_{n\to\infty}\|f_n\|_{L^\infty\RM} = 0$, we have $\lim_{n\to\infty}\|f_n\|_Y = 0$.
    \end{enumerate}
\end{theorem}
\begin{proof}
    Assume first that $X\acInfty Y$. By Proposition~\ref{prop:necessaryCond_fundFuncs_decay}, we have $\lim_{t\to\infty}\fundMin_X(t) = \infty$. Let $\{f_n\}_{n = 1}^\infty$ be a bounded sequence in $X$ such that $\lim_{n\to\infty}\|f_n\|_{L^\infty\RM} = 0$. We need to show that $\lim_{n\to\infty}\|f_n\|_Y = 0$. Without loss of generality, we may assume that $\|f_n\|_X\leq1$ for every $n\in\N$. Let $\varepsilon > 0$. Since $X\acInfty Y$ and $\|f_n\|_X\leq1$ for each $n\in\N$, there exists $a_0>0$ and sets $E_n\subseteq\RR$ such that $\mu(E_n)\leq a_0$ and
    \begin{equation*}
        \|f_n\chi_{\RR\setminus E_n}\|_Y \leq \varepsilon \quad \text{for every $n\in\N$}.
    \end{equation*}
    Hence,
    \begin{equation*}
        \|f_n\|_Y\leq C_Y\big( \|f_n\chi_{E_n}\|_Y + \|f_n\chi_{\RR\setminus E_n}\|_Y \big) \leq C_Y\|f_n\|_{L^\infty\RM}\fundMax_Y(a_0) + C_Y\varepsilon
    \end{equation*}
    for every $n\in\N$. Since $\lim_{n\to\infty}\|f_n\|_{L^\infty\RM} = 0$, it follows that $\lim_{n\to\infty}\|f_n\|_Y = 0$.

    Assume now that (ii) holds. We will show that $X\acInfty Y$ by contradiction. Suppose that $X\acInfty Y$ does not hold. Then there exist $\varepsilon_0 > 0$, a sequence $\{f_n\}_{n=1}^\infty\subseteq\MRM$ with $\|f_n\|_X\leq1$ for every $n\in\N$, and an increasing sequence $\{a_n\}_{n=1}^\infty\subseteq(0,\infty)$ with $\lim_{n\to\infty} a_n = \infty$ such that
    \begin{equation}\label{E:characterization_infty_convergence:1}
        \inf_{\mu(E)\leq a_n} \|f_n\chi_{\RR\setminus E}\|_Y \geq \varepsilon_0 \quad \text{for every $n\in\N$}.
    \end{equation}
    Passing to a subsequence if necessary, we may additionally assume that
    \begin{equation}\label{E:characterization_infty_convergence:2}
        \fundMin_X(a_{n+1}) > \fundMin_X(a_n) \quad \text{for every $n\in\N$},
    \end{equation}
    using the fact that $\lim_{t\to\infty}\fundMin_X(t) = \infty$. For each $n\in\N$, $n\geq2$, define
    \begin{equation}\label{E:characterization_infty_convergence:6}
        E_n = \Big\{ x\in\RR: |f_n(x)| > \frac1{\fundMin_X(a_{n-1})} \Big\}.
    \end{equation}
    We claim that
    \begin{equation}\label{E:characterization_infty_convergence:3}
        \mu(E_n)\leq a_n \quad \text{for every $n\in\N$, $n\geq2$}.
    \end{equation}
    Indeed, if $\mu(E_n) > a_n$ for some $n\in\N$, $n\geq2$, then, using \eqref{E:characterization_infty_convergence:6} in the third inequality,
    \begin{equation*}
        1\geq \|f_n\|_X \geq \|f_n\chi_{E_n}\|_X \geq \frac{\|\chi_{E_n}\|_X}{\fundMin_X(a_{n-1})} \geq \frac{\fundMin_X(a_n)}{\fundMin_X(a_{n-1})} > 1,
    \end{equation*}
    which would be a contradiction; the last inequality follows from \eqref{E:characterization_infty_convergence:2}. Define
    \begin{equation*}
        g_n = f_n\chi_{\RR\setminus E_n},\ n\in\N,\ n\geq2.
    \end{equation*}
    We have
    \begin{equation*}
        \|g_n\|_X \leq \|f_n\|_X\leq1 \quad \text{for every $n\in\N$, $n\geq2$},
    \end{equation*}
    and
    \begin{equation*}
       \lim_{n\to\infty} \|g_n\|_{L^\infty\RM} \leq \lim_{n\to\infty} \frac1{\fundMin_X(a_{n-1})} = 0.
    \end{equation*}
    Consequently, the validity of (ii) implies that
    \begin{equation}\label{E:characterization_infty_convergence:4}
        \lim_{n\to\infty} \|g_n\|_Y = 0.
    \end{equation}
    However, by \eqref{E:characterization_infty_convergence:3} and \eqref{E:characterization_infty_convergence:1}, we also have
    \begin{equation}\label{E:characterization_infty_convergence:5}
        \|g_n\|_Y = \|f_n\chi_{\RR\setminus E_n}\|_Y \geq \inf_{\mu(E)\leq a_n} \|f_n\chi_{\RR\setminus E}\|_Y \geq \varepsilon_0
    \end{equation}
    for every $n\in\N$, $n\geq2$. Hence, in view of \eqref{E:characterization_infty_convergence:4} and \eqref{E:characterization_infty_convergence:5}, we have reached a contradiction.
\end{proof}

\begin{remark}
    Note that the assumption $\lim_{t\to\infty}\fundMin_X(t) = \infty$ in Theorem~\ref{thm:characterization_infty_convergence}(ii) is essential and cannot be omitted. For example, consider $X=Y=L^\infty$. Then every (bounded) sequence $\{f_n\}_{n=1}^\infty\subseteq X$ such that $\lim_{n\to\infty} \|f_n\|_{L^\infty\RM} = 0$ clearly satisfies $\lim_{n\to\infty} \|f_n\|_Y = 0$, but $X\acInfty Y$ does not hold (see Proposition~\ref{prop:necessaryCond_fundFuncs_decay}(i) or Proposition~\ref{prop:X_acInfty_X_never}). This is in contrast to the characterization of $\acLoc$ by means of convergence of functions (see~Theorem~\ref{thm:characterization_loc_ac_convergence} below), or of $\ac$ for that matter (see~\cite[Theorem~3.1]{S:12}), which does not entail any such condition.
\end{remark}

\begin{proposition}\label{prop:X_acInfty_X_never}
There is no admissible q-BFS $X$ such that $X \acInfty X$.
\end{proposition}
\begin{proof}
Suppose that $X \acInfty X$. By Proposition~\ref{prop:necessaryCond_fundFuncs_decay}(i), we have $\lim_{t\to\infty}\fundMin_{X}(t) = \infty$. It follows that there exists a sequence of sets $\{E_n\}_{n = 1}^\infty\subseteq\RR$ such that
\begin{equation*}
\lim_{n\to\infty} \mu(E_n) = \infty
\end{equation*}
and
\begin{equation*}
\fundMin_X(\mu(E_n)) \leq \|\chi_{E_n}\|_X \leq 2\fundMin_X(\mu(E_n)) \quad \text{for every $n\in\N$}.
\end{equation*}
Now, consider the sequence of functions $\{f_n\}_{n = 1}^\infty$ defined by
\begin{equation*}
f_n = \frac{\chi_{E_n}}{2\fundMin_X(\mu(E_n))},\ n\in\N.
\end{equation*}
Then
\begin{equation*}
\sup_{n\in\N} \|f_n\|_X\leq1 \qquad \text{and} \qquad \lim_{n\to\infty} \|f_n\|_{L^\infty\RM} = 0,
\end{equation*}
while
\begin{equation*}
\inf_{n\in\N}\|f_n\|_X \geq \frac1{2}.
\end{equation*}
This contradicts Theorem~\ref{thm:characterization_infty_convergence}, which completes the proof.
\end{proof}

We now prove a characterization of $\acLoc$ by means of convergence of functions. Its comparison with the respective characterization of the relation $\ac$ proved in \cite[Theorem~3.1]{S:12} reveals the key difference between them, which is the root of why it never holds that $X\ac Y$ when $\mu(\RR)=\infty$, whereas the relation $\acLoc$ is still useful.
\begin{theorem}\label{thm:characterization_loc_ac_convergence}
    Let $X$ and $Y$ be q-BFSs. The following two statements are equivalent.
    \begin{enumerate}[(i)]
        \item We have $X\acLoc Y$.
        \item For every bounded sequence $\{f_n\}_{n = 1}^\infty$ in $X$ such that $\lim_{n\to\infty}\mu(\spt f_n) = 0$, we have $\lim_{n\to\infty}\|f_n\|_Y = 0$.
    \end{enumerate}
\end{theorem}
\begin{proof}
    First, assume that $X\acLoc Y$. Let $\{R_n\}_{n = 1}^\infty\subseteq\RR$ be such that $\bigcup_{n=1}^\infty R_n = \RR$ and $0<\mu(R_n)<\infty$ for every $n\in\N$. Set
    \begin{equation*}
       g = \sum_{n=1}^\infty \frac{\chi_{R_n}}{2^nC_Y^n\fundMax_Y(\mu(R_n))}.
    \end{equation*}
    We clearly have $g>0$ $\mu$-a.e.~on $\RR$. Furthermore, it is not hard to see that
    \begin{equation}\label{E:characterization_loc_ac_convergence:1}
        \|g\|_Y\leq1.
    \end{equation}
		Indeed, for every $N\in\N$, using the fact that $C_Y\geq1$, we obtain
		\begin{align*}
			\Big\| \sum_{n=1}^N \frac{\chi_{R_n}}{2^nC_Y^n\fundMax_Y(\mu(R_n))} \Big\|_Y &\leq C_Y \Bigg( \Big\| \sum_{n=2}^N \frac{\chi_{R_n}}{2^nC_Y^n\fundMax_Y(\mu(R_n))} \Big\|_Y + \big\| \frac{\chi_{R_1}}{2C_Y\fundMax_Y(\mu(R_1))} \Big\|_Y \Bigg) \\
			&\leq \Big\| \sum_{n=2}^N \frac{\chi_{R_n}}{2^nC_Y^{n-1}\fundMax_Y(\mu(R_n))} \Big\|_Y + \frac1{2} \\
			&\leq \Big\| \sum_{n=3}^N \frac{\chi_{R_n}}{2^nC_Y^{n-2}\fundMax_Y(\mu(R_n))} \Big\|_Y + \frac1{4} + \frac1{2} \\
			&\leq \sum_{n=1}^N \frac1{2^n},
		\end{align*}
		whence \eqref{E:characterization_loc_ac_convergence:1} follows by letting $N\to\infty$ and using property~(P3) of $\|\cdot\|_Y$.
		
    Now let $\{f_n\}_{n = 1}^\infty$ be a bounded sequence in $X$ such that $\lim_{n\to\infty}\mu(\spt f_n) = 0$. Without loss of generality, we may assume that $\|f_n\|_X\leq1$ for every $n\in\N$. Fix $\varepsilon>0$ and set
    \begin{equation*}
        E_n = \{x\in\RR: |f_n(x)|> \varepsilon g(x)\},\ n\in\N.
    \end{equation*}
    Then
    \begin{equation}\label{E:characterization_loc_ac_convergence:2}
        \mu(E_n)\leq \mu(\spt f_n) \to 0 \quad \text{as $n\to\infty$}.
    \end{equation}
    Using \eqref{E:characterization_loc_ac_convergence:1}, we obtain
    \begin{equation*}
        \|f_n\|_{Y} \leq C_Y\big( \|f_n\chi_{E_n}\|_Y + \|f_n\chi_{\RR\setminus E_n}\|_Y \big) \leq C_Y\sup_{\|f\|_X\leq1}\sup_{\mu(E)\leq \mu(E_n)} \|f\chi_E\|_Y + C_Y\varepsilon
    \end{equation*}
    for every $n\in\N$. Since the double supremum on the right-hand side goes to $0$ as $n\to\infty$ by $X\acLoc Y$ and \eqref{E:characterization_loc_ac_convergence:2}, we conclude that
    \begin{equation*}
        \lim_{n\to\infty}\|f_n\|_Y = 0.
    \end{equation*}

    Second, assume that (ii) holds and suppose that $X\acLoc Y$ does not. Then there exist $\varepsilon_0>0$ and sequences $\{f_n\}_{n=1}^\infty\subseteq\MRM$ with $\|f_n\|_X\leq1$ and $\{E_n\}_{n = 1}^\infty\subseteq\RR$ such that
    \begin{equation}\label{E:characterization_loc_ac_convergence:3}
        \|f_n\chi_{E_n}\|_Y \geq \varepsilon_0 \quad \text{for every $n\in\N$}
    \end{equation}
    and
    \begin{equation*}
        \lim_{n\to\infty} \mu(E_n) = 0.
    \end{equation*}
    Set
    \begin{equation*}
        g_n = f_n\chi_{E_n},\ n\in\N.
    \end{equation*}
   Then $\|g_n\|_X \leq1$ for every $n\in\N$ and $\lim_{n\to\infty}\mu(\spt g_n) = 0$. Hence, by (ii),
    \begin{equation*}
        \lim_{n\to\infty}\|f_n\chi_{E_n}\|_Y = \lim_{n\to\infty}\|g_n\|_Y = 0.
    \end{equation*}
    This, however, contradicts \eqref{E:characterization_loc_ac_convergence:3}, completing the proof.
\end{proof}

We now provide a general compactness principle, which shows how the relations $\acLoc$ and $\acInfty$ can be used together for obtaining compactness (see~Theorem~\ref{thm:application_Lions_compactness_inhom_Sob} and Remark~\ref{rem:application_Lions_compactness_inhom_Sob_general_ri} for its application to inhomogeneous Sobolev spaces on $\rn$).
\begin{theorem}\label{thm:Lions_compactness_lemma}
Let $X,Y$ and $Z$ be q-BFSs, with $X$ being admissible. Assume that $X\acInfty Z$ and $Y\acLoc Z$. Let $\{u_n\}_{n=1}^\infty\subseteq\MRM$ be a sequence of functions that is bounded in both $X$ and $Y$, and satisfies
\begin{equation}\label{E:Lions_compactness_lemma:convergence_in_measure}
    \lim_{n\to\infty} \mu(\{x\in\RR: |u_n(x)| > \varepsilon\}) = 0 \quad \text{for every $\varepsilon > 0$}.
\end{equation}
Then
\begin{equation}\label{E:Lions_compactness_lemma:convergence_in_norm}
    \lim_{n\to\infty}\|u_n\|_Z = 0.
\end{equation}
\end{theorem}
\begin{proof}
    Clearly, we may assume without loss of generality that $\max\{\|u_n\|_X, \|u_n\|_Y\}\leq 1$ for every $n\in\N$. Let $\varepsilon>0$. Since $X\acInfty Z$, there exists $a>0$ such that for each $n\in\N$ we can find a set $E_n\subseteq\RR$ with $\mu(E_n)\leq a$ satisfying
    \begin{equation}\label{E:Lions_compactness_lemma:1}
        \|u_n\chi_{\RR\setminus E_n}\|_Z \leq \varepsilon \quad \text{for every $n\in\N$}.
    \end{equation}
    Next, choose $\kappa>0$ small enough so that
    \begin{equation}\label{E:Lions_compactness_lemma:2}
        \kappa \fundMax_Z(a) \leq \varepsilon,
    \end{equation}
    and define
    \begin{equation*}
        F_n = \{x\in E_n: |u_n(x)| > \kappa\} \quad \text{for every $n\in\N$}.
    \end{equation*}
    Using \eqref{E:Lions_compactness_lemma:1} and \eqref{E:Lions_compactness_lemma:2}, we obtain
    \begin{align}
        \|u_n\chi_{\RR\setminus F_n}\|_Z &\leq C_Z\big(\|u_n\chi_{\RR\setminus E_n}\|_Z + \|u_n\chi_{E_n\setminus F_n}\|_Z \big) \nonumber\\
        &\leq C_Z\big( \varepsilon + \kappa \fundMax_Z(\mu(E_n)) \big)\leq 2C_Z\varepsilon\label{E:Lions_compactness_lemma:3}
    \end{align}
    for every $n\in\N$. Since $Y\acLoc Z$, there exists $\delta>0$ such that
    \begin{equation}\label{E:Lions_compactness_lemma:4}
        \sup_{\mu(F)\leq \delta}\|u_n\chi_F\|_Z \leq \varepsilon \quad \text{for every $n\in\N$}.
    \end{equation}
    Now, for any $\lambda>0$, note that
    \begin{equation*}
        \fundMin_X(\mu(\{x\in\RR: |u_n(x)| > \lambda\})) \leq \frac1{\lambda}\|u_n\|_X \leq \frac1{\lambda} \quad \text{for every $n\in\N$}.
    \end{equation*}
    It follows that
    \begin{equation}\label{E:Lions_compactness_lemma:5}
        \mu(\{x\in\RR: |u_n(x)| > \lambda\}) \leq F\Big( \frac1{\lambda} \Big) \quad \text{for every $n\in\N$},
    \end{equation}
    where $F\colon [0, \infty) \to [0, \infty]$ is the right-continuous generalized inverse function to the nondecreasing function $\fundMin_X$, that is,
    \begin{equation}\label{E:Lions_compactness_lemma:10}
        F(t) = \sup\{s\geq0: \fundMin_X(s) \leq t\},\ t\in[0,\infty).
    \end{equation}
    Observe that $t\leq F(\fundMin_X(t))$ for every $t\in[0, \infty)$. Furthermore, since $X$ is admissible, we have $0 < \fundMin_X(t) < \infty$ for every $t\in(0, \infty)$, and therefore
    \begin{equation}\label{E:Lions_compactness_lemma:6}
        \lim_{t\to0_+} F(t) = 0.
    \end{equation}
    Hence, in view of \eqref{E:Lions_compactness_lemma:5} and \eqref{E:Lions_compactness_lemma:6}, there exists $\lambda>0$ sufficiently large so that
    \begin{equation}\label{E:Lions_compactness_lemma:7}
        \mu(G_n) \leq \delta \quad \text{for every $n\in\N$},
    \end{equation}
    where
    \begin{equation*}
        G_n = \{x\in F_n: |u_n(x)| > \lambda\},\ n\in\N.
    \end{equation*}
    Combining \eqref{E:Lions_compactness_lemma:3}, \eqref{E:Lions_compactness_lemma:4}, and \eqref{E:Lions_compactness_lemma:7}, we obtain
    \begin{align}
        \|u_n\|_Z &\leq C_Z^2\Big( \|u_n\chi_{\RR\setminus F_n}\|_Z + \|u_n\chi_{F_n\setminus G_n}\|_Z + \|u_n\chi_{G_n}\|_Z \Big) \nonumber\\
        &\leq C_Z^2\Big( 2C_Z\varepsilon + \varepsilon + \|u_n\chi_{F_n\setminus G_n}\|_Z \Big) \label{E:Lions_compactness_lemma:8}
    \end{align}
    for every $n\in\N$.

    Finally, since
    \begin{equation*}
        F_n\setminus G_n \subseteq \{x\in\RR: \kappa < |u_n(x)| \leq \lambda\} \quad \text{for every $n\in\N$},
    \end{equation*}
    we have
    \begin{equation}\label{E:Lions_compactness_lemma:9}
        \|u_n\chi_{F_n\setminus G_n}\|_Z \leq \lambda \fundMax_Z(\mu(F_n\setminus G_n)) \quad \text{for every $n\in\N$}.
    \end{equation}
    By Proposition~\ref{prop:necessaryCond_fundFuncs_decay}(ii) and Remark~\ref{rem:necessaryCond_fundFuncs_decay_ri}, we know that the assumption $Y\acLoc Z$ implies $\lim_{a\to0_+} \fundMax_Z(a) = 0$. Since both $\lambda>0$ and $\kappa>0$ are fixed, it follows from \eqref{E:Lions_compactness_lemma:9} and \eqref{E:Lions_compactness_lemma:convergence_in_measure} that
    \begin{equation*}
        \lim_{n\to\infty} \|u_n\chi_{F_n\setminus G_n}\|_Z = 0.
    \end{equation*}
    At last, combining this with \eqref{E:Lions_compactness_lemma:8}, we obtain \eqref{E:Lions_compactness_lemma:convergence_in_norm}.
\end{proof}

When both $X$ and $Y$ are rearrangement invariant, the relations $\acLoc$ and $\acInfty$ have straightforward characterizations by means of nonincreasing functions and representation spaces on the interval $(0, \infty)$ (cf.~\cite[Section~4]{P:26}).
\begin{theorem}\label{thm:characterization_for_ri}
    Let $X$ and $Y$ be r.i.~q-BFSs.
    \begin{enumerate}[(i)]
    \item $X\acInfty Y$ holds if and only if
\begin{equation}\label{E:characterization_for_ri:infty}
    \lim_{a\to \infty} \sup_{\|f\|_{\widebar{X}}\leq1} \|f^*\chi_{(a, \infty)}\|_{\widebar{Y}} = 0.
\end{equation}
\item $X\acLoc Y$ holds if and only if
\begin{equation}\label{E:characterization_for_ri:zero}
    \lim_{a\to 0_+} \sup_{\|f\|_{\widebar{X}}\leq1} \|f^*\chi_{(0, a)}\|_{\widebar{Y}} = 0.
\end{equation}
    \end{enumerate}
\end{theorem}
\begin{proof}
We start with the first equivalence. Assume that $X\acInfty Y$ holds. Let $f\in\M(0,\infty)$ with $\|f\|_{\widebar{X}}\leq1$. Then there exists $\tilde{f}\in\MRM$ such that $\tilde{f}^*=f^*$ (see~\cite[Chapter~2, Corollary~7.8]{BS}). Let $\{s_n\}_{n=1}^\infty\subseteq\Mpl\RM$ be a sequence of nonnegative simple functions satisfying $s_n\nearrow |\tilde{f}|$ $\mu$-a.e. as $n\to\infty$. Consequently, we also have $s_n^*\nearrow \tilde{f}^*=f^*$ as $n\to\infty$. Moreover, it is straightforward to verify that
\begin{equation}\label{E:characterization_for_ri:9}
    (s_n^*\chi_{(a, \infty)})^*\leq (s_n\chi_{\RR\setminus E})^* \quad \text{for all $n\in\N$ and $E\subseteq \RR$ with $\mu(E) \leq a$}.
\end{equation}
Indeed, on the one hand, using the monotonicity of $s_n^*$, we have
\begin{equation*}
		(s_n^*\chi_{(a, \infty)})^*(t) = s_n^*(t + a) \quad \text{for every $t\in[0, \infty)$};
\end{equation*}
on the other hand, it is easy to see that
\begin{equation*}
	\mu( \{x\in\RR: s_n(x)\chi_{\RR\setminus E}(x) > \lambda \} ) \geq \mu( \{x\in\RR: s_n(x) > \lambda \} ) - a
\end{equation*}
for every $\lambda>0$, whence it follows that
\begin{equation*}
(s_n\chi_{\RR\setminus E})^*(t) \geq s_n^*(t + a) \quad \text{for every $t\in[0, \infty)$}.
\end{equation*}
Having established \eqref{E:characterization_for_ri:9}, we immediately obtain
\begin{equation}\label{E:characterization_for_ri:10}
\|(s_n^*\chi_{(a, \infty)})^*\|_{\widebar{Y}}\leq \inf_{\mu(E)\leq a}\|(s_n\chi_{\RR\setminus E})^*\|_{\widebar{Y}} \quad \text{for every $n\in\N$}.
\end{equation}
As an aside, we remark that this is in fact an equality. Hence, using \eqref{E:characterization_for_ri:10}, we have
\begin{align}
    \|s_n^*\chi_{(a, \infty)}\|_{\widebar{Y}} &= \|(s_n^*\chi_{(a, \infty)})^*\|_{\widebar{Y}}\leq \inf_{\mu(E)\leq a}\|(s_n\chi_{\RR\setminus E})^*\|_{\widebar{Y}} \nonumber\\
		&= \inf_{\mu(E)\leq a}\|s_n\chi_{\RR\setminus E}\|_Y \leq \inf_{\mu(E)\leq a}\|\tilde{f}\chi_{\RR\setminus E}\|_Y \label{E:characterization_for_ri:5}
\end{align}
for every $n\in\N$. Since $s_n^*\nearrow f^*$ as $n\to\infty$ and
\begin{equation*}
    \|\tilde{f}\|_X = \|\tilde{f}^*\|_{\widebar{X}} = \|f^*\|_{\widebar{X}} = \|f\|_{\widebar{X}}\leq1,
\end{equation*}
\eqref{E:characterization_for_ri:5}, together with $X\acInfty Y$, yields \eqref{E:characterization_for_ri:infty}. 

For the converse implication, assume that \eqref{E:characterization_for_ri:infty} holds. Let $f\in\MRM$ be such that $\|f\|_X\leq1$, and let $a>0$. We now distinguish two cases. First, assume that $f^*(a) = 0$ and set
\begin{equation*}
    F = \{x\in\RR: |f(x)| > 0\}.
\end{equation*}
Then $\mu(F) \leq a$, and so
\begin{equation}\label{E:characterization_for_ri:2}
    \inf_{\mu(E)\leq a}\|f\chi_{\RR\setminus E}\|_Y \leq \|f\chi_{\RR\setminus F}\|_Y = 0 = \|f^*\chi_{(a,\infty)}\|_{\widebar{Y}}.
\end{equation}
Second, assume that $f^*(a) > 0$ and define
\begin{equation*}
    F_1 = \{x\in\RR: |f(x)| > f^*(a)\} \qquad \text{and} \qquad F_2 = \{x\in\RR: |f(x)| = f^*(a)\}.
\end{equation*}
We have 
\begin{equation}\label{E:characterization_for_ri:1}
    \mu(F_1) = |\{t>0: f^*(t) > f^*(a)\}| = \inf\{t>0: f^*(t) = f^*(a)\} \leq a
\end{equation}
and
\begin{align}
    \mu(F_2) &= \mu(\{x\in\RR: |f(x)| \geq f^*(a)\}) - \mu(F_1)  \nonumber\\
    &= \lim_{n\to\infty}\mu\Big( \Big\{ x\in\RR: |f(x)| > f^*(a) - \frac1{n} \Big\} \Big) - \mu(F_1) \nonumber\\
    &= \lim_{n\to\infty}\Big| \Big\{t>0: f^*(t) > f^*(a) - \frac1{n} \Big\} \Big| - \mu(F_1) \nonumber\\
    &= |\{t>0: f^*(t) \geq f^*(a) \}| - \mu(F_1) \nonumber\\
    &= \sup\{t>0: f^*(t) = f^*(a)\} - \inf\{t>0: f^*(t) = f^*(a)\},\label{E:characterization_for_ri:3}
\end{align}
where we used the fact that $\mu\Big( \Big\{ x\in\RR: |f(x)| > f^*(a) - \frac1{n} \Big\} \Big) < \infty$ for every sufficiently large $n\in\N$. Indeed, for each $a>0$, we observe that
\begin{equation*}
    \frac{\fund_Y(a)}{\fund_X(2a)} = \frac1{\fund_X(2a)} \|\chi_{(0,2a)}\chi_{(a, \infty)}\|_{\widebar{Y}} \leq \sup_{\|f\|_{\widebar{X}}\leq1} \|f^*\chi_{(a, \infty)}\|_{\widebar{Y}},
\end{equation*}
whence it follows that $\lim_{t\to\infty}\fund_X(t) = \infty$ thanks to \eqref{E:characterization_for_ri:infty}. Therefore, if for some $n\in\N$ with $f^*(a) - \frac1{n} > 0$ we had $\mu\Big( \Big\{ x\in\RR: |f(x)| > f^*(a) - \frac1{n} \Big\} \Big) = \infty$, then
\begin{equation*}
    1\geq \|f\|_X \geq \Big( f^*(a) - \frac1{n} \Big) \|\chi_{\{x\in\RR: |f(x)| > f^*(a) - \frac1{n}\}}\|_X = \infty,
\end{equation*}
which would be a contradiction. From \eqref{E:characterization_for_ri:1} and \eqref{E:characterization_for_ri:3} we have $a - \mu(F_1) \leq \mu(F_2)$. Therefore, we can find a set $F_3\subseteq F_2$ such that
\begin{equation*}
    \mu(F_3) = a - \mu(F_1).
\end{equation*}
Setting
\begin{equation*}
    F = F_1 \cup F_3,
\end{equation*}
we obtain
\begin{equation*}
    \mu(F) = \mu(F_1) + \mu(F_3) = a.
\end{equation*}
Moreover, observe that the functions $f\chi_{\RR\setminus F}$ and $f^*\chi_{(a,\infty)}$ are equimeasurable. Hence,
\begin{equation}\label{E:characterization_for_ri:4}
    \inf_{\mu(E)\leq a}\|f\chi_{\RR\setminus E}\|_Y \leq \|f\chi_{\RR\setminus F}\|_Y  =  \|f^*\chi_{(a,\infty)}\|_{\widebar{Y}}.
\end{equation}
Combining \eqref{E:characterization_for_ri:2}, \eqref{E:characterization_for_ri:4}, and using $\|f\|_X = \|f^*\|_{\widebar{X}}\leq1$, we conclude
\begin{equation*}
    \sup_{\|f\|_X\leq1}\inf_{\mu(E)\leq a}\|f\chi_{\RR\setminus E}\|_Y \leq \sup_{\|h\|_{\widebar{X}}\leq1}\|h^*\chi_{(a,\infty)}\|_{\widebar{Y}}.
\end{equation*}
Therefore, the validity of \eqref{E:characterization_for_ri:infty} implies $X\acInfty Y$, which completes the proof of the first equivalence.

We now turn our attention to the second equivalence. Assume that $X\acLoc Y$. Let $f\in\M(0,\infty)$ be such that $\|f\|_{\widebar{X}}\leq1$, and let $a>0$. Then there is a function $g\in\MRM$ such that
\begin{equation}\label{E:characterization_for_ri:6}
    g^* = f^*\chi_{(0,a)}.
\end{equation}
Note that
\begin{equation}\label{E:characterization_for_ri:7}
    \|g\|_X = \|g^*\|_{\widebar{X}} = \|f^*\chi_{(0,a)}\|_{\widebar{X}}\leq 1.
\end{equation}
Furthermore, since $\mu(\{x\in\RR: g(x)\neq0\}) = |\{t>0:g^*(t)\neq0\}|$, we have $\mu(\{x\in\RR: g(x)\neq0\})\leq a$. Therefore, setting
\begin{equation*}
    E = \{x\in\RR: g(x)\neq0\},
\end{equation*}
we see that
\begin{equation}\label{E:characterization_for_ri:8}
    \mu(E)\leq a \qquad \text{and} \qquad g=g\chi_E.
\end{equation}
Hence, using \eqref{E:characterization_for_ri:6}, \eqref{E:characterization_for_ri:8}, and \eqref{E:characterization_for_ri:7}, we obtain
\begin{equation*}
    \|f^*\chi_{(0, a)}\|_{\widebar{Y}} = \|g^*\|_{\widebar{Y}} = \|g\|_Y = \|g\chi_E\|_Y \leq \sup_{\|h\|_X\leq1} \sup_{\mu(F)\leq a} \|h\chi_F\|_Y,
\end{equation*}
from which the validity of \eqref{E:characterization_for_ri:zero} follows. Finally, we prove the converse implication, which concludes the proof. Assume that \eqref{E:characterization_for_ri:zero} holds. Let $f\in\MRM$ and $E\subseteq\RR$ be such that $\|f\|_X\leq1$ and $\mu(E)\leq a$. Note that
\begin{equation*}
    (f\chi_E)^* = (f\chi_E)^*\chi_{(0,a)}.
\end{equation*}
Using this together with the fact that $\|(f\chi_E)^*\|_{\widebar{X}}\leq\|f\|_X\leq1$, we obtain
\begin{equation*}
    \|f\chi_E\|_Y = \|(f\chi_E)^*\|_{\widebar{Y}} = \|(f\chi_E)^*\chi_{(0,a)}\|_{\widebar{Y}} \leq \sup_{\|h\|_{\widebar{X}}\leq1} \|h^*\chi_{(0,a)}\|_{\widebar{Y}}.
\end{equation*}
It follows that \eqref{E:characterization_for_ri:zero} implies the validity of $X\acLoc Y$, as desired.
\end{proof}

By combining Theorem~\ref{thm:characterization_for_ri} with \eqref{prel:fund_ineq}, we obtain simple yet often useful sufficient conditions for the validity of $X\acInfty Y$ and $X\acLoc Y$ when $X$ and $Y$ are r.i.~q-BFSs.
\begin{corollary}\label{cor:suff_cond_for_qri}
Let $X$ and $Y$ be r.i.~q-BFSs. If
\begin{equation*}
    \lim_{a\to\infty} \Big\| \frac1{\fund_X}\chi_{(a, \infty)} \Big\|_{\widebar{Y}} = 0,
\end{equation*}
then $X\acInfty Y$. Similarly, if
\begin{equation*}
    \lim_{a\to0_+} \Big\| \frac1{\fund_X}\chi_{(0,a)} \Big\|_{\widebar{Y}} = 0,
\end{equation*}
then $X\acLoc Y$.
\end{corollary}

When $X$ and $Y$ are BFSs, it is well known that for the ``non-localized'' relation $\ac$, the fact that $X\ac Y$ implies $Y'\ac X'$ (see, e.g.,~\cite[Theorem~7.11.3]{PKJFbook}). Moreover, it can be readily checked that the proof carries over verbatim to the setting of q-BFSs satisfying property~(P5). In fact, by a straightforward modification of the argument, one can show that the same holds for our ``localized'' relation $\acLoc$; that is, $X\acLoc Y$ implies $Y'\acLoc X'$ for q-BFSs $X$ and $Y$ satisfying property~(P5). The proof of this fact is rather simple because both the definition of the associate norm \eqref{prel:asoc_norm} and that of the relation $\acLoc$ involve suprema whose order can be easily exchanged. However, the situation becomes more complicated for the relation $\acInfty$ because its definition involves a combination of a supremum and an infimum. Nevertheless, we have the following result.
\begin{proposition}\label{prop:duality}
    Let $X$ and $Y$ be q-BFSs satisfying property~(P5). Assume that $\lim_{a\to\infty} \fundMin_{Y'}(a) = \infty$. Then $X \acInfty Y$ implies $Y' \acInfty X'$.
\end{proposition}
\begin{proof}
First, note that $X'$ and $Y'$ are BFSs, since $X$ and $Y$ satisfy property~(P5). Assume that $X \acInfty Y$. Let $\varepsilon>0$. Choose $a_1>0$ sufficiently large so that
\begin{equation}\label{E:duality:1}
    \sup_{\|f\|_X\leq1} \inf_{\mu(E)\leq a_1} \|f\chi_{\RR\setminus E}\|_Y \leq \frac{\varepsilon}{2}.
\end{equation}
Let $F$ denote the right-continuous generalized inverse of the function $\fundMin_{Y'}$ (recall~\eqref{E:Lions_compactness_lemma:10}). Since $\lim_{a\to\infty} \fundMin_{Y'}(a) = \infty$, we have $F(t)<\infty$ for every $t>0$. Define
\begin{equation}\label{E:duality:9}
    a = F\Big( \frac{\fundMax_{X'}(a_1)}{\varepsilon} \Big) < \infty.
\end{equation}
For $g\in\MRM$, denote by $E_g$ the set
\begin{equation}\label{E:duality:5}
    E_g = \Big\{ x\in\RR: |g(x)| > \frac{\varepsilon}{\fundMax_{X'}(a_1)} \Big\}.
\end{equation}
Arguing as in \eqref{E:Lions_compactness_lemma:5}, we obtain
\begin{equation}\label{E:duality:2}
    \mu(E_g) \leq F\Big( \frac{\fundMax_{X'}(a_1)}{\varepsilon} \Big) = a \qquad \text{for every $g\in\MRM$ with $\|g\|_{Y'}\leq1$}.
\end{equation}
Now, let $g\in\MRM$ be such that $\|g\|_{Y'}\leq1$. Then
\begin{align}
    \sup_{\|g\|_{Y'}\leq1}\inf_{\mu(E)\leq a}\|g\chi_{\RR\setminus E}\|_{X'} &\leq \sup_{\|g\|_{Y'}\leq1}\|g\chi_{\RR\setminus E_g}\|_{X'} \nonumber\\
    &= \sup_{\substack{\|g\|_{Y'}\leq1\\ \|f\|_X\leq1}} \int_{\RR} |f(x)||g(x)|\chi_{\RR\setminus E_g}(x) \dd\mu(x) \label{E:duality:6}
\end{align}
using \eqref{E:duality:2} and \eqref{prel:asoc_norm}. By \eqref{E:duality:1}, for every $f\in\MRM$ with $\|f\|_{X}\leq1$, there exists a set $G_f$ such that
\begin{equation}\label{E:duality:3}
    \mu(G_f)\leq a_1
\end{equation}
and
\begin{equation}\label{E:duality:4}
    \|f\chi_{\RR\setminus G_f}\|_{Y} \leq \varepsilon.
\end{equation}
Now, for all $\|g\|_{Y'}\leq1$ and $\|f\|_X\leq1$, we have
\begin{align}
    \int_{\RR} |f(x)||g(x)|\chi_{\RR\setminus E_g}(x) \dd\mu(x) &= \int_{\RR} |f(x)||g(x)|\chi_{\RR\setminus E_g}(x)\chi_{\RR\setminus G_f}(x) \dd\mu(x) \nonumber\\
    &\quad+ \int_{\RR} |f(x)||g(x)|\chi_{\RR\setminus E_g}(x)\chi_{G_f}(x) \dd\mu(x) \nonumber\\
    &\leq \|f\chi_{\RR\setminus G_f}\|_{Y} \|g\|_{Y'} + \|f\|_{X} \|g\chi_{\RR\setminus E_g}\chi_{G_f}\|_{X'} \nonumber\\
    &\leq \varepsilon + \|g\chi_{\RR\setminus E_g}\chi_{G_f}\|_{X'} \label{E:duality:7}
\end{align}
by the H\"older inequality \eqref{prel:Holder} and \eqref{E:duality:4}. Furthermore, using \eqref{E:duality:5} and \eqref{E:duality:3}, we obtain
\begin{equation}\label{E:duality:8}
    \|g\chi_{\RR\setminus E_g}\chi_{G_f}\|_{X'} \leq \frac{\varepsilon}{\fundMax_{X'}(a_1)} \fundMax_{X'}(a_1) = \varepsilon.
\end{equation}
Hence, combining \eqref{E:duality:6}, \eqref{E:duality:7}, and \eqref{E:duality:8}, we conclude
\begin{equation*}
    \sup_{\|g\|_{Y'}\leq1}\inf_{\mu(E)\leq a}\|g\chi_{\RR\setminus E}\|_{X'} \leq 2\varepsilon,
\end{equation*}
from which $Y' \acInfty X'$ follows.
\end{proof}

\begin{corollary}\label{cor:duality_for_Y_ri}
Let $X$ and $Y$ be q-BFSs satisfying property~(P5). Assume that the family of simple averaging operators $\{A_E\}_E$ is uniformly bounded on $X$ and that $Y$ is rearrangement invariant. Then $X \acInfty Y$ implies $Y' \acInfty X'$.
\end{corollary}
\begin{proof}
In view of Proposition~\ref{prop:duality}, it suffices to show that $\lim_{a\to\infty}\fundMin_{Y'}(a) = \lim_{a\to\infty}\fund_{Y'}(a) = \infty$. To this end, note that
\begin{equation}\label{E:duality_for_Y_ri:1}
    \lim_{a\to\infty} \frac{\fund_Y(a)}{\fundMin_X(2a)} = 0
\end{equation}
by Proposition~\ref{prop:necessaryCond_fundFuncs_decay}(i). Using Proposition~\ref{prop:E:fund_X_and_X'_quasi}, we have
\begin{equation*}
    \fundMin_{X'}(2a) \leq C \frac{a}{\fundMax_X(2a)} \quad \text{for every $a\in(0, \infty)$},
\end{equation*}
where $C>0$ is a constant independent of $a$, and
\begin{equation*}
    \fund_{Y'}(a) \geq \frac{a}{\fund_Y(a)}  \quad \text{for every $a\in(0, \infty)$}.
\end{equation*}
Hence,
\begin{equation*}
    \frac{\fundMin_{X'}(2a)}{\fund_{Y'}(a)} \leq C \frac{\fund_Y(a)}{\fundMax_X(2a)} \leq C \frac{\fund_Y(a)}{\fundMin_X(2a)} \quad \text{for every $a\in(0, \infty)$}.
\end{equation*}
Consequently,
\begin{equation*}
   \lim_{a\to \infty} \frac{\fundMin_{X'}(2a)}{\fund_{Y'}(a)} = 0,
\end{equation*}
by \eqref{E:duality_for_Y_ri:1}. It follows that $\lim_{a\to\infty}\fund_{Y'}(a) = \infty$, which completes the proof.
\end{proof}

\begin{remark}\label{rem:duality_acInfty_bothRI}
If both $X$ and $Y$ are r.i.~BFSs, then $X \acInfty Y$ if and only if $Y' \acInfty X'$. This follows from \eqref{prel:X_is_double_assoc}, Corollary~\ref{cor:duality_for_Y_ri}, and Remark~\ref{rem:simple_aver_unif_bd_on_ri}.
\end{remark}

The space $L^\infty$ plays a special role in the theory of (rearrangement-invariant) (quasi-)Banach function spaces because its properties are often in a sense extremal and exceptional. We conclude this section with a theorem characterizing the validity of $X\acInfty L^\infty$, but first we need an auxiliary definition (cf.~\cite[Section~5.3]{P:22}).
\begin{definition}
Let $X$ and $Y$ be q-BFSs. We write $X\weakacInfty Y$ if there exists $a_0\in(0,\infty)$ such that
\begin{equation*}
    \sup_{\|f\|_X\leq 1} \inf_{\mu(E)\leq a_0} \|f\chi_{\RR\setminus E}\|_Y < \infty.
\end{equation*}
\end{definition}

\begin{remark}\label{rem:embInfty}
    Clearly, if $X\acInfty Y$, then $X\weakacInfty Y$. The converse implication, however, is false (recall Proposition~\ref{prop:X_acInfty_X_never}), as $X\weakacInfty X$ clearly always holds. Moreover, one can easily verify that if $X\acInfty Y$ and $Y\weakacInfty Z$ (or $X\weakacInfty Y$ and $Y\acInfty Z$), then $X\acInfty Z$. Furthermore, by a suitable modification of the proof of Theorem~\ref{thm:characterization_for_ri}(i), it is not difficult to see that $X\weakacInfty Y$ if and only if
    \begin{equation*}
    \sup_{\|f\|_{\widebar{X}}\leq1} \|f^*\chi_{(a_0, \infty)}\|_{\widebar{Y}} < \infty \quad \text{for some/any $a_0>0$},
    \end{equation*}
    provided that $X$ and $Y$ are rearrangement invariant. Finally, by suitably modifying the proof of Proposition~\ref{prop:duality} (essentially, by replacing $\varepsilon$ in~\eqref{E:duality:9} with a sufficiently large constant $M>0$ so that the function $F$ at the considered point is finite), one can show that
    \begin{equation}\label{E:X_emb_infty_Y_dual}
        X\weakacInfty Y \quad\text{implies}\quad Y'\weakacInfty X'
    \end{equation}
    for all q-BFSs $X$ and $Y$ satisfying property~(P5), whether $\lim_{t\to\infty}\fundMin_{Y'}(t)=\infty$ or not.
\end{remark}

\begin{theorem}\label{thm:char_L_inf_alt}
Let $X$ be an admissible q-BFS. The following two statements are equivalent.
\begin{enumerate}[(i)]
    \item $X\acInfty L^\infty\RM$.
    \item $\lim_{t\to\infty}\fundMin_X(t)=\infty$.
\end{enumerate}
Furthermore, either implies
\begin{enumerate}[(i)]\setcounter{enumi}{2}
    \item $L^\infty\RM \weakacInfty X$ does not hold,
\end{enumerate}
which in turn implies
\begin{enumerate}[(i)]\setcounter{enumi}{3}
    \item $\lim_{t\to\infty}\fundMax_X(t)=\infty$.
\end{enumerate}
Finally, if $X$ is rearrangement invariant, all four statements are equivalent to each other, with (ii) and (iv) reading as
\begin{equation*}
    \lim_{t\to\infty}\fund_X(t) = \infty.
\end{equation*}
\end{theorem}
\begin{proof}
The equivalence of (i) and (ii) follows from Proposition~\ref{prop:necessaryCond_fundFuncs_decay} and Theorem~\ref{thm:characterization_infty_convergence}.

Next, assume that (i) holds. Suppose that (iii) is not true. Consequently, we simultaneously have $X\acInfty L^\infty\RM$ and $L^\infty\RM \weakacInfty X$. It follows that $X\acInfty X$. This, however, contradicts Proposition~\ref{prop:X_acInfty_X_never}. Hence, (i) implies (iii).


Finally, we show that (iii) implies (iv). Suppose that (iv) fails, which means that we have $\lim_{t\to\infty}\fundMax_X(t) < \infty$. Consequently,
\begin{equation*}
    \|\chi_{\RR}\|_X < \infty.
\end{equation*}
Hence,
\begin{equation*}
    \sup_{\|f\|_{L^\infty\RM}\leq1} \inf_{\mu(E)\leq 1} \|f\chi_{\RR\setminus E}\|_X \leq \|\chi_{\RR}\|_X < \infty.
\end{equation*}
It follows that $L^\infty\RM \weakacInfty X$ holds. In other words, (iii) does not hold.
\end{proof}

\begin{remark}
Note that we never have $L^\infty\RM\acInfty X$ in view of Proposition~\ref{prop:necessaryCond_fundFuncs_decay}(i).
\end{remark}

\section{Uniform decay and endpoint spaces}\label{sec:endpoints}
There always exist the smallest and largest rearrangement-invariant Banach function spaces, the so-called \emph{endpoint spaces}, having a prescribed fundamental function. These spaces play an important role in the theory of r.i.~BFSs (see, e.g., \cite{BS,KPS}). This section is devoted to the study of the relation $\acInfty$ in connection with the endpoint spaces. Although their theory can be extended, to some extent, to the setting of r.i.~q-BFSs (see~\cite[Section~4]{MNPT:25}), we restrict our attention to BFSs for simplicity's sake. Furthermore, we also study only the relation $\acInfty$ because similar results for the relation $\acLoc$ can be easily deduced from those contained in \cite{S:12} for the non-localized relation $\ac$.

Given a quasiconcave function $\varphi\colon [0, \infty) \to [0, \infty)$, the \emph{Marcinkiewicz space} $M_\varphi\RM$ is defined via the rearrangement-invariant Banach function norm
\begin{equation*}
    \|f\|_{M_\varphi\RM} = \sup_{t\in(0, \infty)} f^{**}(t)\varphi(t),\ f\in\Mpl\RM.
\end{equation*}
The fundamental function of $M_\varphi\RM$ satisfies $\fund_{M_\varphi} = \varphi$. For example, consider $\varphi(t) = t^{1/p}$, $t\in(0, \infty)$, for $p\in[1,\infty]$. When $p\in(1, \infty]$, $M_\varphi$ is the Lorentz space $L^{p,\infty}$, up to equivalence of norms (see~Section~\ref{sec:examples} for more information on Lorentz spaces). When $p=1$, $M_\varphi$ coincides with $L^1$.

Every r.i.~BFS $X$ satisfies the endpoint embedding
\begin{equation}\label{prel:endpoint_embedding_M}
     X \hookrightarrow M_{\fund_X}\RM,
\end{equation}
where $\hookrightarrow$ denotes a continuous embedding\textemdash that is, for (quasi-)normed spaces $A$ and $B$, we write $A\hookrightarrow B$ if $A\subseteq B$ and there exists a constant $C>0$ such that $\|x\|_{B} \leq C\|x\|_A$ for every $x\in A$. The embedding \eqref{prel:endpoint_embedding_M} holds with $C = 1$.

Every r.i.~BFS $X$ can be equivalently renormed by another rearrangement-invariant Banach function norm so that its new fundamental function is concave. Moreover, this can be done in such a way that the new fundamental function coincides with the least concave majorant of the original fundamental function. The least concave majorant $\tilde{\varphi}\colon[0, \infty) \to [0, \infty)$ of a quasiconcave function $\varphi$ satisfies
\begin{equation}\label{prel:concave_majorant}
\frac{1}{2} \tilde{\varphi} \leq \varphi \leq \tilde{\varphi}.
\end{equation}
Therefore, we may assume, without loss of generality, that the fundamental function of an r.i.~BFS is concave whenever convenient.

Assuming $\varphi\colon[0, \infty) \to [0, \infty)$ is concave and nontrivial, the \emph{Lorentz endpoint space} $\Lambda_\varphi\RM$ is defined by means of the rearrangement-invariant Banach function norm
\begin{equation*}
    \|f\|_{\Lambda_\varphi\RM} = \int_0^\infty f^*(t) \dd{\varphi(t)} = \|f\|_{L^\infty\RM}\varphi(0_+) + \int_0^\infty f^*(t) \varphi'(t) \dd{t}
\end{equation*}
for any $f\in\Mpl\RM$. We have $\fund_{\Lambda_\varphi} = \varphi$. For example, consider once more $\varphi(t) = t^{1/p}$, $t\in(0, \infty)$, for $p\in[1,\infty]$. When $p\in[1, \infty)$, $\Lambda_\varphi$ is the Lorentz space $L^{p,1}$, up to a constant multiple of the norm. When $p=\infty$, $\Lambda_\varphi$ coincides with $L^\infty$.

Every r.i.~BFS $X$ satisfies the endpoint embedding
\begin{equation*}
    \Lambda_{\tilde{\fund}_X}\RM \hookrightarrow X,
\end{equation*}
where $\tilde{\fund}_X$ is the least concave majorant of $\fund_X$. Furthermore,
\begin{equation}\label{prel:endpoint_assoc}
(\Lambda_{\varphi})'\RM = M_{\tilde{\psi}}\RM,    
\end{equation}
where $\tilde{\psi}$ denotes the least concave majorant of the function $\psi(t) = t/\varphi(t)$, $t\in(0, \infty)$.

\begin{theorem}\label{thm:characterization_Marcz_endpoints}
Let $X$ be an r.i.~BFS. Assume that
\begin{equation}\label{E:characterization_Marcz_endpoints:fund_cond}
	\lim_{t\to \infty}\frac{t}{\varphi_X(t)} = \infty.
\end{equation}
Let $\varphi\colon[0, \infty) \to [0, \infty)$ be a nontrivial concave function. Denote by $\psi$ the least concave majorant of the function $t\mapsto t/\varphi(t)$. Then the following three statements are equivalent.
\begin{enumerate}[(i)]
\item $M_\varphi\RM \acInfty X$;
\item $\lim_{a\to\infty} \|f^*\chi_{(a,\infty)}\|_{\widebar{X}} = 0$ for every $f\in M_\varphi(0,\infty)$;
\item $\lim_{a\to\infty} \|\psi'\chi_{(a,\infty)}\|_{\widebar{X}} = 0$.
\end{enumerate}
\end{theorem}

\begin{proof}
The fact that (i) implies (ii) follows from Theorem~\ref{thm:characterization_for_ri}. We now prove that (ii) implies (iii). It suffices to show that $\psi'\in M_\varphi(0,\infty)$. Using \eqref{prel:concave_majorant} together with the nonnegativity and concavity of $\psi$, we obtain
\begin{align}
\|\psi'\|_{M_\varphi(0,\infty)} &= \sup_{t\in(0,\infty)} \frac{\varphi(t)}{t}\int_0^t\psi'(s)\dd{s} \leq \sup_{t\in(0,\infty)}\frac{\varphi(t)}{t}\psi(t) \nonumber\\
&\leq \sup_{t\in(0,\infty)}\frac{\varphi(t)}{t}\frac{2t}{\varphi(t)} < \infty. \label{E:characterization_Marcz_endpoints:4}
\end{align}

Finally, we prove that (iii) implies (i). By \eqref{prel:endpoint_assoc} and Remark~\ref{rem:duality_acInfty_bothRI}, the statement~(i) is valid if and only if
\begin{equation}\label{E:characterization_Marcz_endpoints:1}
    X' \acInfty \Lambda_{\psi}\RM.
\end{equation}
Moreover, by \eqref{prel:fund_X_and_asocX_for_ri}, the condition \eqref{E:characterization_Marcz_endpoints:fund_cond} is equivalent to
\begin{equation*}
    \lim_{t\to \infty}\varphi_{X'}(t) = \infty.
\end{equation*}
Thus, in view of Theorem~\ref{thm:characterization_infty_convergence}, to establish \eqref{E:characterization_Marcz_endpoints:1} it is sufficient to show that
\begin{equation}\label{E:characterization_Marcz_endpoints:2}
    \lim_{n\to\infty} \|f_n\|_{\Lambda_{\psi}\RM} = 0
\end{equation}
for every sequence $\{f_n\}_{n = 1}^\infty\subseteq X'$ such that
\begin{equation}\label{E:characterization_Marcz_endpoints:3}
    \|f_n\|_{X'}\leq 1 \quad \text{for every $n\in\N$} \qquad \text{and} \qquad \lim_{n\to\infty} \|f_n\|_{L^\infty\RM} = 0.
\end{equation}
Let $\{f_n\}_{n = 1}^\infty\subseteq X'$ satisfy \eqref{E:characterization_Marcz_endpoints:3}. Fix $a>0$. Then
\begin{align*}
    \int_a^\infty f_n^*(t) \psi'(t) \dd{t} \leq \|f_n^*\|_{\widebar{X}'} \|\psi'\chi_{(a, \infty)}\|_{\widebar{X}} \leq \|\psi'\chi_{(a, \infty)}\|_{\widebar{X}} \quad \text{for every $n\in\N$}
\end{align*}
by the H\"{o}lder inequality \eqref{prel:Holder} and \eqref{E:characterization_Marcz_endpoints:3}. Therefore,
\begin{align*}
\|f_n\|_{\Lambda_{\psi}\RM} &= \|f_n\|_{L^\infty\RM}\psi(0_+) + \int_0^a f_n^*(t) \psi'(t) \dd{t} + \int_a^\infty f_n^*(t) \psi'(t) \dd{t}\\
&\leq \|f_n\|_{L^\infty\RM} \psi(0_+) + \|f_n\|_{L^\infty\RM} \int_0^a \psi'(t) \dd{t} + \|\psi'\chi_{(a, \infty)}\|_{\widebar{X}}\\
&= \|f_n\|_{L^\infty\RM}\psi(a) + \|\psi'\chi_{(a, \infty)}\|_{\widebar{X}}
\end{align*}
for every $n\in\N$. Using (iii) together with \eqref{E:characterization_Marcz_endpoints:3}, we obtain \eqref{E:characterization_Marcz_endpoints:2}, which concludes the proof.
\end{proof}

\begin{remark}
    The assumption~\eqref{E:characterization_Marcz_endpoints:fund_cond} is natural and not restrictive. Indeed, if it is not satisfied, $Z \acInfty X$ does not hold for any r.i.~BFS $Z$. To see this, note that
    \begin{equation*}
    \frac{\fund_X(t)}{\fund_Z(t)} \geq  \frac1{\fund_Z(1)} \frac{\fund_X(t)}{t} \quad \text{for every $t\geq1$}
    \end{equation*}
    thanks to the fact that the function $t\mapsto t/\fund_Z(t)$ is nondecreasing on $(0, \infty)$. Consequently, if \eqref{E:characterization_Marcz_endpoints:fund_cond} is violated, then $Z \acInfty X$ cannot hold by Proposition~\ref{prop:necessaryCond_fundFuncs_decay} (recall also Remark~\ref{rem:necessaryCond_fundFuncs_decay_ri}).
\end{remark}

\begin{theorem}\label{thm:characterization_Lorentz_endpoints}
Let $\varphi,\psi\colon[0, \infty) \to [0, \infty)$ be nontrivial concave functions. Then
\begin{equation}\label{thm:characterization_Lorentz_endpoints:acInfty}
    \Lambda_\varphi\RM \acInfty \Lambda_\psi\RM
\end{equation}
if and only if
\begin{equation}\label{thm:characterization_Lorentz_endpoints:fund}
    \lim_{t\to\infty}\frac{\psi(t)}{\varphi(t)}=0.
\end{equation}
\end{theorem}
\begin{proof}
In view of Proposition~\ref{prop:necessaryCond_fundFuncs_decay} (recall also Remark~\ref{rem:necessaryCond_fundFuncs_decay_ri}), it suffices to show that \eqref{thm:characterization_Lorentz_endpoints:fund} implies \eqref{thm:characterization_Lorentz_endpoints:acInfty}. Assume that \eqref{thm:characterization_Lorentz_endpoints:fund} holds. Let $s\in\Mpl(0, \infty)$ be a nontrivial simple function. Then its nonincreasing rearrangement can
be written as
\begin{equation*}
    s^* = \sum_{j = 1}^N \gamma_j \chi_{[0, b_j)}
\end{equation*}
for some $N\in\N$, with $\gamma_j>0$ and $0<b_1<b_2<\dots < b_N < \infty$. We have
\begin{align}
    \frac{\|s^*\chi_{(a,\infty)}\|_{\Lambda_\psi(0, \infty)}}{\|s\|_{\Lambda_\varphi(0, \infty)}} &= \frac{\|\sum_{b_j>a}\gamma_j\chi_{[0,b_j-a)}\|_{\Lambda_\psi(0, \infty)}}{\|\sum_{j = 1}^N \gamma_j \chi_{[0, b_j)}\|_{\Lambda_\varphi(0, \infty)}}
= \frac{\sum_{b_j>a}\gamma_j\psi(b_j-a)}{\sum_{j=1}^N \gamma_j\varphi(b_j)} \nonumber\\
&\leq \frac{\sum_{b_j>a}\gamma_j\psi(b_j)}{\sum_{j=1}^N \gamma_j\varphi(b_j)}
\leq \frac{\sum_{b_j>a}\gamma_j\varphi(b_j)}{\sum_{j=1}^N \gamma_j\varphi(b_j)} \sup_{t>a} \frac{\psi(t)}{\varphi(t)}
\leq \sup_{t>a} \frac{\psi(t)}{\varphi(t)}. \label{E:characterization_Lorentz_endpoints:1}
\end{align}
Let $f\in\M(0, \infty)$ and fix $a>0$. There exists a sequence $\{s_n\}_{n = 1}^\infty\subseteq\Mpl(0,\infty)$ of nonnegative simple functions such that $s_n \nearrow |f|$ a.e. In particular, we have $s_n^*\chi_{(a,\infty)} \nearrow f^*\chi_{(a,\infty)}$. Using \eqref{E:characterization_Lorentz_endpoints:1}, we obtain
\begin{equation}\label{E:characterization_Lorentz_endpoints:2}
    \frac{\|f^*\chi_{(a,\infty)}\|_{\Lambda_\psi(0, \infty)}}{\|f\|_{\Lambda_\varphi(0, \infty)}} = \lim_{n\to\infty} \frac{\|s_n^*\chi_{(a,\infty)}\|_{\Lambda_\psi(0, \infty)}}{\|s_n\|_{\Lambda_\varphi(0, \infty)}} \leq \sup_{t>a} \frac{\psi(t)}{\varphi(t)}.
\end{equation}
Hence,
\begin{equation*}
    \lim_{a\to\infty}\sup_{\|f\|_{\Lambda_\varphi(0, \infty)}\leq1}\|f^*\chi_{(a,\infty)}\|_{\Lambda_\psi(0, \infty)} = 0,
\end{equation*}
by \eqref{thm:characterization_Lorentz_endpoints:fund} and \eqref{E:characterization_Lorentz_endpoints:2}. Therefore,
\eqref{thm:characterization_Lorentz_endpoints:acInfty} follows from
Theorem~\ref{thm:characterization_for_ri}.
\end{proof}

\begin{theorem}\label{thm:characterization_Marcz_into_Lorentz}
Let $\varphi,\psi\colon[0, \infty) \to [0, \infty)$ be nontrivial concave functions. Then
\begin{equation}\label{E:characterization_Marcz_into_Lorentz:acInfty}
    M_\varphi\RM \acInfty\Lambda_\psi\RM
\end{equation}
if and only if
\begin{equation}\label{E:characterization_Marcz_into_Lorentz:embInfty_and_fund}
    M_\varphi\RM \weakacInfty \Lambda_\psi\RM \qquad \text{and} \qquad \lim_{t\to\infty}\frac{\psi(t)}{\varphi(t)} = 0.
\end{equation}
\end{theorem}
\begin{proof}
By Remark~\ref{rem:embInfty} and Proposition~\ref{prop:necessaryCond_fundFuncs_decay}, we see that \eqref{E:characterization_Marcz_into_Lorentz:acInfty} implies \eqref{E:characterization_Marcz_into_Lorentz:embInfty_and_fund}. Assume that \eqref{E:characterization_Marcz_into_Lorentz:embInfty_and_fund} holds. Note that, using  \eqref{E:characterization_Marcz_into_Lorentz:embInfty_and_fund}, we have
\begin{equation*}
    \lim_{t\to\infty} \frac{t}{\psi(t)}  = \lim_{t\to\infty} \frac{\varphi(t)}{\psi(t)} \frac{t}{\varphi(t)} = \infty.
\end{equation*}
Let $\eta$ be the least concave majorant of $t\mapsto t/\varphi(t)$. By Theorem~\ref{thm:characterization_Marcz_endpoints}, \eqref{E:characterization_Marcz_into_Lorentz:acInfty} will follow once we show that
\begin{equation}\label{E:characterization_Marcz_into_Lorentz:5}
    \lim_{a\to\infty}\|\eta'\chi_{(a, \infty)}\|_{\Lambda_\psi(0, \infty)} = 0.
\end{equation}
To this end, the monotonicity and concavity of $\eta$ imply that
\begin{equation*}
    0\leq\eta'(t)=\lim_{s\to t}\frac{\eta(t)-\eta(s)}{t-s}\leq\frac{\eta(t)}{t} \quad \text{for a.e.~$t\in(0, \infty)$}.
\end{equation*}
Hence,
\begin{equation}\label{E:characterization_Marcz_into_Lorentz:3}
    (\eta')^*(t)\leq \frac{\eta(t)}{t} \quad \text{for every $t\in(0, \infty)$}.
\end{equation}
Using $\eta'\in M_\varphi(0,\infty) \weakacInfty \Lambda_\psi(0,\infty)$ (see~\eqref{E:characterization_Marcz_endpoints:4}) together with Remark~\ref{rem:embInfty}, we obtain
\begin{equation*}
\int_0^\infty \left(\eta'\chi_{(1,\infty)}\right)^*(t)\psi'(t) \dd{t} = \int_0^\infty \eta'(t+1)\psi'(t) \dd{t} < \infty.
\end{equation*}
Therefore,
\begin{equation}\label{E:characterization_Marcz_into_Lorentz:4}
    \lim_{a\to\infty} \int_a^\infty \eta'(t+1)\psi'(t) \dd{t} = 0 
\end{equation}
by the dominated convergence theorem. Now, using the monotonicity of $\eta'$, \eqref{E:characterization_Marcz_into_Lorentz:3}, and \eqref{prel:concave_majorant}, for every $a\geq1$, we obtain
\begin{align*}
\|\eta'\chi_{(a,\infty)}\|_{\Lambda_\psi(0, \infty)} &= \left(\eta'\right)^*(a) \psi(0_+) + \int_0^a \eta'(t+a)\psi'(t) \dd{t} + \int_a^\infty \eta'(t+a) \psi'(t) \dd{t}\\
&\leq  \left(\eta'\right)^*(a) \psi(0_+) + \left(\eta'\right)^*(a) \int_0^a\psi'(t) \dd{t} + \int_a^\infty \eta'(t+1)\psi'(t) \dd{t}\\
&= \left(\eta'\right)^*(a)\psi(a) + \int_a^\infty \eta'(t+1)\psi'(t) \dd{t}\\
&\leq \frac{\eta(a)}{a}\psi(a) + \int_a^\infty \eta'(t+1)\psi'(t) \dd{t} \\
&\leq 2\frac{\psi(a)}{\varphi(a)} + \int_a^\infty \eta'(t+1)\psi'(t) \dd{t}.
\end{align*}
Hence, combining this with \eqref{E:characterization_Marcz_into_Lorentz:embInfty_and_fund} and \eqref{E:characterization_Marcz_into_Lorentz:4}, we see that \eqref{E:characterization_Marcz_into_Lorentz:5} holds, from which \eqref{E:characterization_Marcz_into_Lorentz:acInfty} follows.
\end{proof}

\begin{corollary}
Let $X$ and $Y$ be r.i.~BFSs of type $\Lambda$ or $M$ (not necessarily of the same type). Assume that their fundamental functions are concave. Then
\begin{equation}\label{E:emb_endpoint_spaces:acInfty}
    X \acInfty Y
\end{equation}
if and only if
\begin{equation}\label{E:emb_endpoint_spaces:embInfty_and_fund}
    X \weakacInfty Y \qquad \text{and} \qquad \lim_{t\to\infty}\frac{\fund_Y(t)}{\fund_X(t)} = 0.
\end{equation}
\end{corollary}
\begin{proof}
The fact that \eqref{E:emb_endpoint_spaces:acInfty} implies \eqref{E:emb_endpoint_spaces:embInfty_and_fund} follows from Proposition~\ref{prop:necessaryCond_fundFuncs_decay} and Remark~\ref{rem:embInfty}. For the converse implication, we consider four cases. First, assume that both $X$ and $Y$ are of type $\Lambda$. Then the converse implication follows from Theorem~\ref{thm:characterization_Lorentz_endpoints}. Second, assume that $X=\Lambda_{\fund_X}$ and $Y = M_{\fund_Y}$. This case also follows from Theorem~\ref{thm:characterization_Lorentz_endpoints} combined with the embedding $\Lambda_{\fund_Y}\RM\hookrightarrow M_{\fund_Y}\RM = Y$. Third, assume that $X=M_{\fund_X}$ and $Y = \Lambda_{\fund_Y}$. Then the converse implication follows from Theorem~\ref{thm:characterization_Marcz_into_Lorentz}. Finally, assume that $X$ and $Y$ are of type $M$. By Remark~\ref{rem:duality_acInfty_bothRI} and \eqref{prel:endpoint_assoc}, \eqref{E:emb_endpoint_spaces:acInfty} is equivalent to
\begin{equation}\label{E:emb_endpoint_spaces:1}
    \Lambda_{\tilde{\varphi}}\RM \acInfty \Lambda_{\tilde{\psi}}\RM,
\end{equation}
where $\tilde{\varphi}$ and $\tilde{\psi}$ are the nondecreasing concave majorants of the functions $t\mapsto t/\fund_Y(t)$ and $t\mapsto t/\fund_X(t)$, respectively. Since
\begin{equation*}
    \limsup_{t\to\infty} \frac{\tilde{\psi}(t)}{\tilde{\varphi}(t)} \leq 2\limsup_{t\to\infty} \frac{\fund_Y(t)}{\fund_X(t)}
\end{equation*}
by \eqref{prel:concave_majorant}, \eqref{E:emb_endpoint_spaces:embInfty_and_fund} implies \eqref{E:emb_endpoint_spaces:1} by Theorem~\ref{thm:characterization_Lorentz_endpoints}. Consequently, \eqref{E:emb_endpoint_spaces:acInfty} holds.
\end{proof}

When $\varphi$ is the identity function on $[0, \infty)$, both $\Lambda_\varphi$ and $M_\varphi$ coincide with $L^1$. We conclude this section with this special, important case.
\begin{proposition}
Let $X$ be an r.i.~BFS. Then the following three statements are equivalent.
\begin{enumerate}[(i)]
	\item $L^1\RM \acInfty X$;
  \item $X \weakacInfty L^1\RM$ does not hold;
  \item $\lim_{t\to\infty}\frac{t}{\fund_X(t)} = \infty$.
\end{enumerate}
\end{proposition}
\begin{proof}
By \eqref{E:X_emb_infty_Y_dual}, \eqref{prel:X_is_double_assoc}, and Theorem~\ref{thm:char_L_inf_alt}, the statement (ii) is equivalent to
\begin{equation*}
X' \acInfty L^\infty\RM
\end{equation*}
and to
\begin{equation*}
\lim_{t\to\infty}\varphi_{X'}(t) = \infty.
\end{equation*}
These, in turn, are equivalent to statements (i) and (iii) by Remark~\ref{rem:duality_acInfty_bothRI} and \eqref{prel:fund_X_and_asocX_for_ri}, respectively.
\end{proof}

\section{Examples}\label{sec:examples}
\subsection{Lebesgue and Lorentz spaces}
In this subsection, we characterize the validity of
\begin{equation}\label{E:Lorentz_ac_infty}
    L^{p_1,q_1}\RM \acInfty L^{p_2,q_2}\RM
\end{equation}
between two Lorentz spaces, which also covers the case of Lebesgue spaces (see~\eqref{E:Lpp_is_Lp} below).

The rearrangement invariance of the usual Lebesgue (quasi-)norm $\|\cdot\|_{L^p(\RR, \mu)}$, $p\in(0, \infty]$, follows from the layer cake formula (see, e.g.,~\cite[Chapter~2, Proposition~1.8]{BS} or \cite[Theorem~1.13]{LL:01}). More precisely, we have
\begin{equation}\label{E:Leb_norm_rearrangement}
\|f\|_{L^p(\RR,\mu)} = \|f^*\|_{L^p(0, \infty)} \quad \text{for every $f\in\M(\RR, \mu)$}.
\end{equation}
The fundamental function of the Lebesgue space $L^p$ satisfies $\fund_{L^p}(t) = t^{1/p}$ for every $t\in(0, \infty)$. Moreover, $(L^p)'\RM = L^{p'}\RM$ for $p\in[1, \infty]$ and $(L^p)'\RM = \{0\}$ for $p\in(0,1)$.

\emph{Lorentz spaces} $L^{p, q}(\RR, \mu)$ are an important generalization of Lebesgue spaces. For $p,q\in(0, \infty]$, we define the functional
$\|\cdot\|_{L^{p,q}(\RR, \mu)}$ as
\begin{equation*}
\|f\|_{L^{p,q}(\RR,\mu)} = \|t^{\frac1{p} - \frac1{q}}f^*(t)\|_{L^q(0, \infty)},\ f\in\Mpl(\RR,\mu).
\end{equation*}
When $p=\infty$ and $q\in(0, \infty)$, we have $\|f\|_{L^{p,q}(\RR, \mu)} = \infty$ for every $f\in\MRM\setminus\{0\}$. In all remaining cases (that is, $p<\infty$ or $p=q=\infty$), the functional $\|\cdot\|_{L^{p,q}(\RR, \mu)}$ defines (at least) a rearrangement-invariant quasi-Banach function norm. The fundamental function of the resulting Lorentz space $L^{p,q}(\RR,\mu)$ satisfies
\begin{equation}\label{E:Lorentz_fund}
    \varphi_{L^{p,q}}(t) = c_{p,q}t^{\frac1{p}} \quad \text{for every $t\in[0, \infty)$},
\end{equation}
where $c_{p,q}=(q/p)^{1/q}$, which is to be interpreted as $1$ when $q=\infty$. Note that
\begin{equation}\label{E:Lpp_is_Lp}
    L^{p,p}\RM = L^p\RM \quad \text{with equal norms},
\end{equation}
which follows directly from \eqref{E:Leb_norm_rearrangement}. Furthermore, we have
\begin{equation}\label{E:Lorentz_sec_par_inc}
L^{p, q_1}(\RR,\mu) \hookrightarrow L^{p, q_2}(\RR,\mu) \quad \text{whenever $q_1\leq q_2$}.
\end{equation}
When $1\leq q\leq p < \infty$ or $p=q=\infty$, the functional $\|\cdot\|_{L^{p,q}(\RR, \mu)}$ is a rearrangement-invariant Banach function norm. If $1< p < q \leq\infty$, then $\|\cdot\|_{L^{p,q}(\RR,\mu)}$ itself is not a rearrangement-invariant Banach function norm\textemdash it fails to be subadditive\textemdash but it is equivalent to one (see~\cite[Chapter~4, Section~4]{BS} or \cite{H:66}). More precisely, for $p\in(1, \infty)$ and $q\in[1 ,\infty]$, the functional
\begin{equation*}
\|f\|_{L^{(p,q)}(\RR,\mu)} = \|f^{**}\|_{L^{p,q}(0, \infty)},\ f\in\Mpl(\RR,\mu),
\end{equation*}
defines a rearrangement-invariant Banach function norm, and there exist positive constants $C_1$ and $C_2$ such that
\begin{equation*}
C_1\|f\|_{L^{(p,q)}(\RR,\mu)} \leq \|f\|_{L^{p,q}(\RR,\mu)} \leq C_2\|f\|_{L^{(p,q)}(\RR,\mu)} \quad \text{for every $f\in\Mpl(\RR,\mu)$}.
\end{equation*}
In view of this equivalence, we may (and do) regard the functional $\|\cdot\|_{L^{p,q}(\RR,\mu)}$ as a rearrangement-invariant Banach function norm if (in fact, only if)
\begin{equation}\label{E:Lorentz_ri_space}
    p\in(1, \infty) \quad\text{and}\quad q\in[1, \infty] \qquad\text{or}\qquad p=q=1 \qquad\text{or}\qquad p=q=\infty.
\end{equation}
When \eqref{E:Lorentz_ri_space} holds, the associate space satisfies $(L^{p,q})'\RM = L^{p',q'}\RM$ with equivalent norms. When $q\in(0,1)$ and $p\in[1, \infty)$, the Lorentz space $L^{p,q}(\RR,\mu)$ is a non-normable r.i.~q-BFS that satisfies property (P5), and its associate space is $(L^{p,q})'\RM = L^{p',\infty}\RM$, up to equivalence of norms. Finally, if either $p\in(0,1)$ and $q\in(0, \infty]$, or $p=1$ and $q\in(1, \infty]$, then $L^{p,q}(\RR,\mu)$ is a non-normable r.i.~q-BFS that does not satisfy property (P5), and in this case $(L^{p,q})'\RM = \{0\}$.

\begin{proposition}
    Let $p_1,p_2,q_1,q_2\in(0,\infty]$. For each $j=1,2$, assume that either $p_j$ is finite or both $p_j$ and $q_j$ are infinite. Then \eqref{E:Lorentz_ac_infty} holds if and only if $p_1 < p_2$.
\end{proposition}
\begin{proof}
    Using Proposition~\ref{prop:necessaryCond_fundFuncs_decay} and \eqref{E:Lorentz_fund}, we see that if \eqref{E:Lorentz_ac_infty} holds, then necessarily $p_1<p_2$. Therefore, it suffices to prove that \eqref{E:Lorentz_ac_infty} holds when $p_1<p_2$. We consider two cases. First, assume $p_1<p_2<\infty$. In view of \eqref{E:Lorentz_sec_par_inc}, it is sufficient to show that
    \begin{equation}\label{E:Lorentz_acInfty_char:1}
        L^{p_1,\infty}\RM \acInfty L^{p_2,r}\RM,
    \end{equation}
    where $r = \min\{q_2, 1\}$. To this end, note that
    \begin{align}
        \|f^*\chi_{(a, \infty)}\|_{L^{p_2,r}\RM}^r &\leq \|f^*(a)\chi_{(0,a)} + f^*\chi_{(a, \infty)}\|_{L^{p_2,r}\RM}^r \nonumber\\
        &= \int_0^a t^{\frac{r}{p_2} - 1} f^*(a)^r\dd{t} + \int_a^\infty t^{\frac{r}{p_2} - 1} f^*(t)^r \dd{t} \label{E:Lorentz_acInfty_char:2}
    \end{align}
    for all $f\in\M\RM$ and $a>0$. Using \eqref{prel:fund_ineq} and \eqref{E:Lorentz_fund}, we have
    \begin{align*}
        \int_0^a t^{\frac{r}{p_2} - 1} f^*(a)^r\dd{t} &\leq c_{p_1,\infty}^r a^{-\frac{r}{p_1}}\|f\|_{L^{p_1, \infty}\RM}^r \int_0^a t^{\frac{r}{p_2} - 1} \dd{t} \\
        &= c_{p_1,\infty}^r\frac{p_2}{r} a^{\frac{r}{p_2} - \frac{r}{p_1}} \|f\|_{L^{p_1, \infty}\RM}^r
    \end{align*}
    for all $f\in\M\RM$ and $a>0$. It follows that
    \begin{equation}\label{E:Lorentz_acInfty_char:3}
        \lim_{a\to\infty} \sup_{\|f\|_{L^{p_1, \infty}\RM}\leq1} \int_0^a t^{\frac{r}{p_2} - 1} f^*(a)^r\dd{t} = 0.
    \end{equation}
    Furthermore,
    \begin{align*}
        \int_a^\infty t^{\frac{r}{p_2} - 1} f^*(t)^r \dd{t} &\leq \Big( \sup_{t\in(0, \infty)} t^{\frac1{p_1}} f^*(t) \Big)^r \int_a^\infty t^{\frac{r}{p_2} - \frac{r}{p_1} - 1} \dd{t} \\
        &= \frac{p_1 p_2}{r(p_2 - p_1)} a^{r\frac{p_1 - p_2}{p_1 p_2}} \|f\|_{L^{p_1, \infty}\RM}^r
    \end{align*}
    for all $f\in\M\RM$ and $a>0$. Consequently, we have
    \begin{equation}\label{E:Lorentz_acInfty_char:4}
        \lim_{a\to\infty} \sup_{\|f\|_{L^{p_1, \infty}\RM}\leq1} \int_a^\infty t^{\frac{r}{p_2} - 1} f^*(t)^r \dd{t} = 0.
    \end{equation}
    Hence, combining \eqref{E:Lorentz_acInfty_char:3} and \eqref{E:Lorentz_acInfty_char:4} with \eqref{E:Lorentz_acInfty_char:2}, we obtain
    \begin{equation*}
        \lim_{a\to\infty} \sup_{\|f\|_{L^{p_1, \infty}\RM}\leq1} \|f^*\chi_{(a, \infty)}\|_{L^{p_2,r}\RM} = 0,
    \end{equation*}
    from which \eqref{E:Lorentz_acInfty_char:1} follows using Theorem~\ref{thm:characterization_for_ri}.

    Finally, when $p_1<\infty$ and $p_2=q_2 = \infty$, we have
    \begin{align*}
        \lim_{a\to\infty} \sup_{\|f\|_{L^{p_1, q_1}\RM}\leq1} \|f^*\chi_{(a, \infty)}\|_{L^\infty\RM} &= \lim_{a\to\infty} \sup_{\|f\|_{L^{p_1, q_1}\RM}\leq1} f^*(a) \\
        &\leq c_{p_1,q_1}\lim_{a\to\infty} a^{-\frac1{p_1}} = 0
    \end{align*}
    by \eqref{prel:fund_ineq} and \eqref{E:Lorentz_fund}.
\end{proof}

\begin{remark}
    The validity of \eqref{E:Lorentz_ac_infty} when $p_1<p_2$ can also be derived from Corollary~\ref{cor:suff_cond_for_qri}. This alternative approach requires essentially the same computations. Furthermore, similar computations show that $L^{p_1,q_1}\RM \acLoc L^{p_2,q_2}\RM$ if and only if $p_1 > p_2$\textemdash exactly as expected in view of the known fact that $L^{p_1,q_1}\RM \ac L^{p_2,q_2}\RM$ if and only if $p_1 > p_2$ provided that $\mu(\RR) < \infty$.
\end{remark}

\subsection{Orlicz spaces}
In this subsection, we completely characterize the validity of
\begin{equation}\label{E:Orlicz_ac_infty}
    L^B\RM \acInfty L^A\RM,
\end{equation}
where $L^A\RM$ and $L^B\RM$ are Orlicz spaces (see, e.g., \cite{PKJFbook,RR:91}).

A function \( A\colon [0, \infty] \to [0, \infty] \) is called a \emph{Young function} if it is left-continuous and convex on \( [0, \infty) \), satisfies \( A(0) = 0 \), and is not identically constant on \( (0, \infty) \). Given a Young function $A$, the functional $\|\cdot\|_{L^A\RM}$ defined by
\begin{equation*}
\|f\|_{L^A\RM} = \inf \left\{ \lambda > 0: \int_\RR A\left( \frac{f(x)}{\lambda} \right) \dd{\mu}(x) \leq 1 \right\},\ f\in\Mpl\RM,
\end{equation*}
is a rearrangement-invariant Banach function norm. Moreover,
\begin{equation}\label{E:Orlicz_via_rearrangement}
    \int_\RR A(|f(x)|) \dd{\mu}(x) = \int_0^\infty A(f^*(t)) \dd{t} \quad \text{for every $f\in\M\RM$}.
\end{equation}
The resulting r.i.~BFS is called an \emph{Orlicz space}. For example, if $p\in[1, \infty)$, then $\|\cdot\|_{L^p(\RR,\mu)} = \|\cdot\|_{L^A(\RR,\mu)}$ for $A(t) = t^p$, $t\geq0$. Furthermore, $\|\cdot\|_{L^\infty(\RR,\mu)} = \|\cdot\|_{L^A(\RR,\mu)}$ for the Young function  $A(t) = \infty\cdot\chi_{(1, \infty]}(t)$, $t\geq0$.

For a Young function $A$, we have
\begin{equation}\label{E:Orlicz_fund}
\varphi_{L^A}(t)=\frac{1}{A^{-1}(\frac{1}{t})} \quad \text{for every $t\in[0, \infty)$},
\end{equation}
where
\begin{equation*}
A^{-1}(t) = \sup\{s\geq0: A(s)\leq t\},\ t\in[0, \infty],
\end{equation*}
is the generalized right-continuous inverse of $A$. The function $A^{-1}$ is nondecreasing, continuous on $[0,\infty)$, and satisfies $A^{-1}(0)=t_0$ and $A^{-1}(\infty)=\infty$, where
\begin{equation*}
    t_0 = \sup\{t\geq0:A(t)=0\} \in [0, \infty).
\end{equation*}
Furthermore, for every $f\in \M\RM$,
\begin{equation}\label{E:Orlicz_modular_vs_norm_ball}
\|f\|_{L^A\RM}\leq 1 \quad \text{if and only if} \quad \int_\RR A(|f(x)|)\dd\mu(x) \leq 1.
\end{equation}

Before characterizing the validity of \eqref{E:Orlicz_ac_infty}, we need to introduce a relation between Young functions.
\begin{definition}    
Let $A$ and $B$ be Young functions. We say that \emph{$B$ essentially dominates $A$ near zero}, and write $A\ll_0B$, if $B(t)>0$ for every $t>0$, and 
\begin{equation*}
\lim_{t\to0_+}\frac{A(\lambda t)}{B(t)}=0 \quad \text{for every $\lambda > 0$}.
\end{equation*}
\end{definition}

\begin{remark}
Let $B$ be a Young function. Using \eqref{E:Orlicz_fund} and the fact that $B^{-1}(0) = t_0$, it is straightforward to verify that
\begin{equation}\label{E:orlicz_lim_fund}
B(t)>0 \quad \text{for every $t>0$} \qquad \text{if and only if} \qquad \lim_{t\to\infty} \varphi_{L^B}(t) = \infty.
\end{equation}
\end{remark}

\begin{proposition}\label{prop:Orlicz_acInfty}
Let A and B be Young functions. Then
\begin{equation*}
L^B\RM\acInfty L^A\RM \quad \text{if and only if} \quad A\ll_0 B.
\end{equation*}
\end{proposition}
\begin{proof}
Assume that $A\ll_0B$ holds. In particular, $B(t)>0$ for every $t>0$. Let $\varepsilon>0$ be given. For $\lambda = 1/\varepsilon$, the definition of $A \ll_0 B$
yields $\delta>0$ such that
\begin{equation}\label{E:Orlicz_acInfty:1}
 A\left(\frac{t}{\varepsilon}\right)\leq B(t) \quad \text{for every $t\in[0,\delta)$}.
\end{equation}
Furthermore, by \eqref{prel:fund_ineq} and \eqref{E:orlicz_lim_fund}, there exists $a>0$ such that
\begin{equation*}
\sup_{\|f\|_{L^B(0, \infty)}\leq1} f^*(t) \leq \sup_{\|f\|_{L^B(0, \infty)}\leq1} f^*(a) \leq \frac{1}{\varphi_{L^B}(a)} < \delta \quad \text{for every $t\geq a$}.
\end{equation*}
Consequently, using \eqref{E:Orlicz_acInfty:1}, \eqref{E:Orlicz_via_rearrangement}, and \eqref{E:Orlicz_modular_vs_norm_ball}, we obtain
\begin{align*}
    \int_0^\infty A\left(\frac{f^*(t)\chi_{(a,\infty)}(t)}{\varepsilon}\right)\dd{t} & \leq \int_a^\infty B\left(f^*(t)\right)\dd{t}\leq \int_0^\infty B\left(f^*(t)\right)\dd{t}\\
    &= \int_0^\infty B\left(|f(t)|\right)\dd{t} \leq 1
\end{align*}
for every $f\in\M(0, \infty)$ such that $\|f\|_{L^B(0, \infty)}\leq1$. Hence,
\begin{equation*}
\sup_{\|f\|_{L^B(0,\infty)}\leq1}\|f^*\chi_{(a,\infty)}\|_{L^A(0,\infty)}\leq\varepsilon.
\end{equation*}
Since $\varepsilon>0$ was arbitrary, Theorem~\ref{thm:characterization_for_ri} implies that $L^B\RM\acInfty L^A\RM$.

Now, assume that $L^B\RM\acInfty L^A\RM$ holds. By Proposition~\ref{prop:necessaryCond_fundFuncs_decay}(i) and \eqref{E:orlicz_lim_fund}, we have $\lim_{t\to\infty}\fund_{L^B}(t) = \infty$ and $B(t)>0$ for every $t>0$. Suppose that
\begin{equation*}
\limsup_{t\to0_+}\frac{A(\lambda t)}{B(t)}>0 \quad \text{for some $\lambda>0$}.
\end{equation*}
Since $A\ll_0 B$ if and only if $A\ll_0cB$ for any $c>0$, we may assume without loss of generality that
\begin{equation*}
\limsup_{t\to0_+}\frac{A(\lambda t)}{B(t)}>1.
\end{equation*}
Then we can find a sequence $\{t_n\}_{n=1}^{\infty}\subseteq(0,\infty)$ with $\lim_{n\to\infty}t_n = 0$ such that 
\begin{equation}\label{E:example_orlicz:2}
\infty>A(\lambda t_n)>B(t_n)>0 \quad \text{for every} \quad n\in\mathbb{N}.
\end{equation}
We will reach a contradiction with the second statement in Theorem~\ref{thm:characterization_infty_convergence}. For each $n\in\N$, set $a_n=1/B(t_n)$ and choose a set $E_n\subseteq\RR$ with $0< \mu(E_n)= a_n < \infty$. Define
\begin{equation*}
    f_n=t_n\chi_{E_n} \quad \text{for every $n\in\N$}.
\end{equation*}
Then
\begin{equation*}
\int_\RR B(f_n) \dd\mu = a_n B\left(t_n\right) = 1 \quad \text{for every $n\in\N$}.
\end{equation*}
Hence,
\begin{equation*}
    \|f_n\|_{L^B\RM}\leq1 \quad \text{for every $n\in\N$}.
\end{equation*}
Moreover, since $t_n\to 0_+$, we have
\begin{equation*}
    \lim_{n\to\infty}\|f_n\|_{L^\infty\RM}=0.
\end{equation*}
However, using \eqref{E:example_orlicz:2}, we also have
\begin{equation*}
\int_\RR A(\lambda f_n) \dd\mu > \int_{E_n} B(t_n)\dd\mu = 1 \quad \text{for every $n\in\N$}.
\end{equation*}
It follows that $\|f_n\|_{L^A\RM} > \frac{1}{\lambda}>0$ for every $n\in\N$. Hence,
\begin{equation*}
    \limsup_{n\to\infty}\|f_n\|_{L^A\RM} > 0,
\end{equation*}
which contradicts the second statement in Theorem~\ref{thm:characterization_infty_convergence}. Therefore, $A\ll_0B$ holds as desired.
\end{proof}

\begin{remark}
    When $\mu(\RR) < \infty$, it is known (see, e.g., \cite[Chapter~4]{PKJFbook}) that $L^B\RM \ac L^A\RM$ if and only if $A(t) < \infty$ for every $t\in(0, \infty)$ and
    \begin{equation*}
        \lim_{t\to\infty} \frac{A(\lambda t)}{B(t)} = 0 \quad \text{for every $\lambda > 0$}.
    \end{equation*}
    It is not difficult to verify that the same condition characterizes the validity of $L^B\RM \acLoc L^A\RM$ when $\mu(\RR) = \infty$.
\end{remark}

\subsection{Inhomogeneous Sobolev spaces and compactness}
In this final subsection, we demonstrate how the general compactness principle in Theorem~\ref{thm:Lions_compactness_lemma} can be applied to convergence of functions in (inhomogeneous) Sobolev spaces on the whole space $\rn$, $n\geq2$. We do not aim to provide a comprehensive treatment of this problem here. Instead, we restrict ourselves to a rather concrete, nonlimiting situation, which illustrates a possible use of Theorem~\ref{thm:Lions_compactness_lemma} without requiring additional technical computations.

Let $A$ be a Young function and let either $p\in(1, n)$ and $q\in[1, \infty]$ or $p=q=1$. Denote by $W^1(L^A, L^{p,q})(\rn)$ the (inhomogeneous) Sobolev space
\begin{equation*}
    W^1(L^A, L^{p,q})(\rn) = \{u: \text{$u$ is weakly differentiable}, u\in L^A(\rn), |\nabla u|\in L^{p,q}(\rn) \}
\end{equation*}
endowed with the norm
\begin{equation*}
    \|u\|_{W^1(L^A, L^{p,q})(\rn)} = \|u\|_{L^A(\rn)} + \|\,|\nabla u|\,\|_{L^{p,q}(\rn)},\ u\in W^1(L^A, L^{p,q})(\rn).
\end{equation*}
By \cite[Theorems~3.11 and 4.3]{ACPS:18}, we have the embedding
\begin{equation}\label{E:application_Lions_compactness_inhom_Sob:Sob_emb}
    W^1(L^A, L^{p,q})(\rn) \hookrightarrow \big( L^A\cap L^{\frac{np}{n-p},q}_{loc} \big)(\rn),
\end{equation}
where
\begin{equation*}
    L^{\frac{np}{n-p},q}_{loc}(\rn) = \{f\in\M(\rn): \|f^*\chi_{(0,1)}\|_{L^{\frac{np}{n-p},q}(0, \infty)} < \infty\}
\end{equation*}
is a local version of the Lorentz space $L^{\frac{np}{n-p},q}(\rn)$. Since $1 < \frac{np}{n-p} < \infty$, $L^{\frac{np}{n-p},q}_{loc}(\rn)$ is an r.i.~BFS when $q\in[1, \frac{np}{n-p}]$, and is equivalent to one when $q\in(\frac{np}{n-p}, \infty]$. It is not difficult to verify that a function $f\in\M(\rn)$ belongs to $L^{\frac{np}{n-p},q}_{loc}(\rn)$ if and only if $f\chi_E\in L^{\frac{np}{n-p},q}(\rn)$ for every $E\subseteq\rn$ with finite measure.

\begin{theorem}\label{thm:application_Lions_compactness_inhom_Sob}
    Let $A$ be a Young function, and let either $p\in(1, n)$ and $q\in[1, \infty]$ or $p=q=1$. Let $\{u_k\}_{k=1}^\infty\subseteq W^1(L^A, L^{p,q})(\rn)$ be a bounded sequence such that
    \begin{equation}\label{thm:application_Lions_compactness_inhom_Sob:eq:4}
        \lim_{k\to\infty} |\{x\in\rn: |u_k(x)| > \varepsilon \}| = 0 \quad \text{for every $\varepsilon>0$}.
    \end{equation}
    If $B$ is a Young function such that $B \ll_0 A$ and simultaneously either
    \begin{equation}\label{thm:application_Lions_compactness_inhom_Sob:eq:1}
        \lim_{t\to\infty} \frac{t^{\frac{np}{n-p}}}{B(t)} = \infty \quad \text{when $q\in\Big[ 1, \frac{np}{n-p} \Big]$}
    \end{equation}
    or
    \begin{equation}\label{thm:application_Lions_compactness_inhom_Sob:eq:2}
        \int_1^\infty \Bigg( \frac{B(t)}{t^{\frac{np}{n-p}}} \Bigg)^{\frac{q}{q-{\frac{np}{n-p}}}} \, \frac{\dd{t}}{t} < \infty \quad \text{when $q\in\Big( \frac{np}{n-p}, \infty \Big]$},
    \end{equation}
    then
    \begin{equation}\label{thm:application_Lions_compactness_inhom_Sob:eq:3}
        \lim_{k\to\infty} \|u_k\|_{L^B(\rn)} = 0.
    \end{equation}
    When $q = \infty$, the exponent $\frac{q}{q-{\frac{np}{n-p}}}$ in \eqref{thm:application_Lions_compactness_inhom_Sob:eq:2} is to be interpreted as $1$.
\end{theorem}
\begin{proof}
    It follows from \eqref{E:application_Lions_compactness_inhom_Sob:Sob_emb} that the sequence $\{u_k\}_{k=1}^\infty$ is bounded in both $L^A(\rn)$ and $L^{\frac{np}{n-p},q}_{loc}(\rn)$. By Proposition~\ref{prop:Orlicz_acInfty}, $L^A(\rn)\acInfty L^B(\rn)$, as $B \ll_0 A$. Furthermore, it follows from \cite[Theorem~1.2]{MPT:25} that $L^{\frac{np}{n-p},q}_{loc}(\rn) \acLoc L^B(\rn)$ when either \eqref{thm:application_Lions_compactness_inhom_Sob:eq:1} or \eqref{thm:application_Lions_compactness_inhom_Sob:eq:2} is satisfied. Therefore, using Theorem~\ref{thm:Lions_compactness_lemma} with $X = L^A(\rn)$, $Y = L^{\frac{np}{n-p},q}_{loc}(\rn)$, and $Z = L^B(\rn)$, we obtain \eqref{thm:application_Lions_compactness_inhom_Sob:eq:3}.
\end{proof}

\begin{remark}\label{rem:application_Lions_compactness_inhom_Sob_general_ri}
    Theorem~\ref{thm:application_Lions_compactness_inhom_Sob} is a special case of a more general principle. Let $X$ and $Y$ be r.i.~BFSs over $\rn$ and let $W^1(X,Y)(\rn)$ be the (inhomogeneous) Sobolev space defined as $W^1(L^A, L^{p,q})(\rn)$ but with $L^A$ and $L^{p,q}$ replaced by $X$ and $Y$, respectively. By \cite[Theorem~4.3]{ACPS:18},
    \begin{equation*}
        W^1(X,Y)(\rn) \hookrightarrow (Y_{opt} \cap X)(\rn),
    \end{equation*}
    where $Y_{opt}$ is the r.i.~BFS over $\rn$ whose associate function norm satisfies
    \begin{equation*}
        \|g\|_{\widebar{Y_{opt}}'} = \|t^\frac1{n} (g\chi_{(0,1)})^{**}(t)\chi_{(0,1)}(t)\|_{\widebar{Y}'},\ g\in\Mpl(0, \infty).
    \end{equation*}
    Let $\{u_k\}_{k = 1}^\infty\subseteq W^1(X,Y)(\rn)$ be a bounded sequence satisfying \eqref{thm:application_Lions_compactness_inhom_Sob:eq:4}. It follows from Theorem~\ref{thm:Lions_compactness_lemma} that
    \begin{equation*}
        \lim_{k\to\infty} \|u_k\|_Z = 0
    \end{equation*}
    for every q-BFS $Z$ such that $Y_{opt} \acLoc Z$ and $X\acInfty Z$.
\end{remark}

We conclude this paper with a discussion of the condition \eqref{thm:application_Lions_compactness_inhom_Sob:eq:4}. When dealing with problems where compactness cannot be taken for granted, such as problems involving Sobolev spaces on the whole $\rn$, the (potential) lack of compactness often makes them difficult. Nevertheless, some compactness arguments may still be possible, but they require extra considerations. In view of this, the presence of \eqref{thm:application_Lions_compactness_inhom_Sob:eq:4} (or a similar condition) is natural, because its verification requires such extra considerations. A typical example of when they can often be carried out is when the problem in question involves some symmetry of functions, such as the radial symmetry (see~\cite{BL:80,BL:83,BL:83b,CGM:78}). To give a concrete example, we sketch how to use Theorem~\ref{thm:application_Lions_compactness_inhom_Sob} to obtain the compactness of the restricted Sobolev embedding \eqref{intro:Sob_rad_rn} for $m=1$, $p\in(1, n)$, and $p < q < np/(n-p)$. We omit the case $p=1$ here, because it would entail additional difficulties due to the non-reflexivity of $W^{1,1}(\rn)$, which goes beyond the scope of this discussion and paper. 

Let $\{v_k\}_{k=1}^\infty \subseteq W^{1,p}_R(\rn)$ be a bounded sequence. Using the reflexivity of $W^{1,p}_R(\rn)$, which is a closed subspace of $W^{1,p}(\rn)$, and the compactness of Sobolev embeddings $W^{1,p}(B_r) \hookrightarrow L^1(B_r)$ on balls (here $B_r = \{x\in\rn: |x| < r\}$), we know that there is a function $v\in W^{1,p}_R(\rn)$ and a subsequence of $\{v_k\}_{k=1}^\infty$, which we will still denote $\{v_k\}_{k=1}^\infty$, such that $v_k \rightharpoonup v$ weakly in $W^{1,p}_R(\rn)$ and
\begin{equation}\label{E:application_compactness_Sob_rad_rn:1}
	\lim_{k \to \infty} v_k(x) = v(x) \quad \text{for a.e.~$x\in\rn$}.
\end{equation}
Set
\begin{equation*}
u_k = v_k - v \quad \text{for every $k\in\N$}.
\end{equation*}
We claim that the sequence $\{u_k\}_{k=1}^\infty$, which is bounded in $W^{1,p}(\rn)$, satisfies \eqref{thm:application_Lions_compactness_inhom_Sob:eq:4}. Note that since the measure of $\rn$ is infinite, we cannot simply deduce \eqref{thm:application_Lions_compactness_inhom_Sob:eq:4} from the pointwise convergence \eqref{E:application_compactness_Sob_rad_rn:1}. However, this is exactly the point where the presence of symmetry comes into play.

The radial symmetry ensures that the functions $\{u_k\}_{k=1}^\infty$ have to decay with a certain rate, which was initially proved by Strauss in \cite{S:77} and later generalized in various ways by Lions in \cite{L:82}. Namely, we have
\begin{equation}\label{E:application_compactness_Sob_rad_rn:2}
 |u_k(x)| \leq C \|u_k\|_{W^{1,p}(\rn)} |x|^{\frac{1 - n}{p}} \quad \text{for a.e.~$x\in\rn$},
\end{equation}
where $C>0$ is a constant depending only on the dimension $n\geq2$ and $p$. Note that the right-hand side of \eqref{E:application_compactness_Sob_rad_rn:2} goes uniformly in $k$ to $0$ as $|x|\to\infty$, since $\{u_k\}_{k=1}^\infty$ is bounded in $W^{1,p}(\rn)$. Now, suppose that $\{u_k\}_{k=1}^\infty$ does not satisfy \eqref{thm:application_Lions_compactness_inhom_Sob:eq:4}, that is, there is $\varepsilon_0>0$ such that
\begin{equation*}
\limsup_{k\to\infty} |\{x\in\rn: |u_k(x)| > \varepsilon_0 \}| > 0.
\end{equation*}
At the same time, \eqref{E:application_compactness_Sob_rad_rn:2} tells us that there is a large enough $r_0>0$ such that
\begin{equation*}
\{x\in\rn: |u_k(x)| > \varepsilon_0 \} \subseteq B_{r_0} \quad \text{for every $k\in\N$}.
\end{equation*}
Now, on the one hand, we have
\begin{equation*}
\limsup_{k\to\infty} |\{x\in B_{r_0}: |u_k(x)| > \varepsilon_0 \}| = \limsup_{k\to\infty} |\{x\in\rn: |u_k(x)| > \varepsilon_0 \}| > 0.
\end{equation*}
On the other hand, \eqref{E:application_compactness_Sob_rad_rn:1} and the dominated convergence theorem imply that
\begin{equation*}
\lim_{k\to\infty} |\{x\in B_{r_0}: |u_k(x)| > \varepsilon_0 \}| = 0,
\end{equation*}
which is a contradiction. Hence, the sequence $\{u_k\}_{k=1}^\infty$ satisfies \eqref{thm:application_Lions_compactness_inhom_Sob:eq:4}.

Finally, having verified \eqref{thm:application_Lions_compactness_inhom_Sob:eq:4}, we may use Theorem~\ref{thm:application_Lions_compactness_inhom_Sob} (with $A(t) = t^p$, both $p$ and $q$ there equal to $p$ here, and $B(t) = t^q$) to obtain
\begin{equation*}
\lim_{k\to\infty} \|u_k\|_{L^q(\rn)} = 0
\end{equation*}
for every $p < q < np/(n-p)$.

\section*{Acknowledgment}
We thank the referees for their careful reading of the manuscript and their valuable comments, which helped us improve the paper.

\end{document}